\documentclass[11pt, a4paper]{amsart}
\usepackage{amssymb}
\usepackage{amsmath}
\usepackage{amssymb}
\usepackage{amsthm}
\usepackage{bbm}
\usepackage{bm}
\usepackage{mathtools}
\usepackage{mathrsfs}
\usepackage[T1]{fontenc}
\usepackage[usenames,dvipsnames,svgnames,table]{xcolor}
\usepackage{esint}
\usepackage{soul}
\allowdisplaybreaks
\usepackage{combelow}
\usepackage{enumitem}
\setenumerate{label={\rm (\alph{*})}}
\usepackage[colorlinks=true,citecolor=black,urlcolor=black,
linkcolor=black]{hyperref}

\usepackage{graphicx}
\usepackage{tikz}

\usepackage{geometry}
\usepackage{verbatim}

\newtheorem{theorem}{Theorem}[section]
\newtheorem{bigthm}{Theorem}

\newtheorem{lemma}[theorem]{Lemma}
\newtheorem{proposition}[theorem]{Proposition}
\newtheorem{remo}[theorem]{Remark}
\newtheorem{corollary}[theorem]{Corollary}

\newtheorem{definition}[theorem]{Definition}

\theoremstyle{remark}

\newtheorem{example}[theorem]{Example}
\numberwithin{equation}{section}

\usepackage{booktabs}

\normalsize

\def\XXint#1#2#3{{\setbox0=\hbox{$#1{#2#3}{\int}$ }
		\vcenter{\hbox{$#2#3$ }}\kern-.6\wd0}}

		\makeatletter
\newcommand*\bigcdot{\mathpalette\bigcdot@{.5}}
\newcommand*\bigcdot@[2]{\mathbin{\vcenter{\hbox{\scalebox{#2}{$\m@th#1\bullet$}}}}}
\makeatother

\newcommand{\di}{\operatorname{div}}

\newcommand{\curl}{\operatorname{curl}}
\newcommand{\dif}{\operatorname{d}\!}
\newcommand{\spt}{\operatorname{spt}}

\renewcommand{\geq}{\geqslant}

\newcommand{\id}{\operatorname{Id}}

\newcommand{\lin}{\operatorname{Lin}}
\newcommand{\locc}{\operatorname{loc}}
\renewcommand{\leq}{\leqslant}

\newcommand{\proj}{\mathrm{Proj}}
\newcommand{\rank}{\operatorname{rank}}

\newcommand{\rmim}{\mathrm{im\,}}

\newcommand{\spann}{\operatorname{span}}
\newcommand{\sym}{\operatorname{sym}}

\newcommand{\ww}{\mathbf w}

\newcommand{\A}{\mathscr{A}}
\newcommand{\B}{\mathscr{B}}

\newcommand{\C}{\mathscr{C}}

\newcommand{\EE}{\mathscr{E}}
\newcommand{\F}{\mathcal{F}}
\newcommand{\FT}{\mathscr{F}}

\newcommand{\Le}{\mathscr{L}}

\newcommand{\N}{\mathbb{N}}
\newcommand{\PP}{\mathscr{P}}
\newcommand{\R}{\mathbb{R}}

\makeatletter
\DeclareFontFamily{OMX}{MnSymbolE}{}
\DeclareSymbolFont{MnLargeSymbols}{OMX}{MnSymbolE}{m}{n}
\SetSymbolFont{MnLargeSymbols}{bold}{OMX}{MnSymbolE}{b}{n}
\DeclareFontShape{OMX}{MnSymbolE}{m}{n}{
    <-6>  MnSymbolE5
   <6-7>  MnSymbolE6
   <7-8>  MnSymbolE7
   <8-9>  MnSymbolE8
   <9-10> MnSymbolE9
  <10-12> MnSymbolE10
  <12->   MnSymbolE12
}{}
\DeclareFontShape{OMX}{MnSymbolE}{b}{n}{
    <-6>  MnSymbolE-Bold5
   <6-7>  MnSymbolE-Bold6
   <7-8>  MnSymbolE-Bold7
   <8-9>  MnSymbolE-Bold8
   <9-10> MnSymbolE-Bold9
  <10-12> MnSymbolE-Bold10
  <12->   MnSymbolE-Bold12
}{}

\let\llangle\@undefined
\let\rrangle\@undefined
\DeclareMathDelimiter{\llangle}{\mathopen}%
                     {MnLargeSymbols}{'164}{MnLargeSymbols}{'164}
\DeclareMathDelimiter{\rrangle}{\mathclose}%
                     {MnLargeSymbols}{'171}{MnLargeSymbols}{'171}
\makeatother

\newcommand{\abs}[1]{\left\lvert #1 \right\rvert}
\newcommand{\abss}[1]{\lvert #1 \rvert}
\newcommand{\brac}[1]{\left(#1\right)}

\newcommand{\brangle}[1]{\left\langle#1\right\rangle}
\newcommand{\Brangle}[1]{\langle\kern-2.5pt\langle #1 \rangle\kern-2.5pt\rangle}

\newcommand{\col}{\colon\,}

\newcommand{\indi}{\mathbbm{1}}
\newcommand{\norm}[1]{\left\lVert#1\right\rVert}
\newcommand{\normm}[1]{\lVert#1\rVert}
\newcommand{\normiii}[1]{{\left\vert\kern-0.25ex\left\vert\kern-0.25ex\left\vert #1 \right\vert\kern-0.25ex\right\vert\kern-0.25ex\right\vert}}

\newcommand{\pprime}{\prime\prime}

\newcommand{\ssubset}{\subset\joinrel\subset}

\newcommand{\weakto}{\rightharpoonup}

\AtEndDocument{\bigskip{\footnotesize%
    \noindent\textsc{Max Planck Institute for Mathematics in the Sciences,\\ Inselstrasse 22, Leipzig, 04103, Germany}\\
    \textit{E-mail address}, Z.~Li: \texttt{zhuolin.li@mis.mpg.de}
		
		\vspace{2.4pt}
		\noindent\textsc{Department of Mathematics and Statistics, Georgetown University,
                  \\3700 Reservoir Road NW, Washington, D.C., 20007 , USA }   \\
		\noindent\textit{E-mail address}, B.~Rai\cb{t}\u{a}: \texttt{br607@georgetown.edu}
}}

\usepackage[
backend=biber,
style=alphabetic,
maxnames=99
]{biblatex}

\begin{document}
	\title[Regularity for $\A$-quasiconvex integrals]{Partial regularity and higher integrability \\for $\A$-quasiconvex variational problems}
\author[Z. Li]{Zhuolin Li}
\author[B. Rai\cb{t}\u{a}]{Bogdan Rai\cb{t}\u{a}}

\maketitle
\begin{center}
  \small{  \textit{Dedicated to Jan Kristensen on the occasion of his 60th birthday}}
\end{center}
\begin{abstract}
 We prove that minimizers of variational problems on open sets $\Omega\subset\R^n$
 $$ \mbox{minimize}\quad
\mathcal E(v)=\int_\Omega f(v(x))\mathrm{d} x\quad\text{for } \mathscr{A}
v=0, $$
  are partially continuous provided that the integrands $f$ are strongly
$\mathscr{A}$-quasiconvex in a suitable sense. We consider $p$-growth problems
with $1<p<\infty$, linear constant rank PDE operators $\mathscr{A}$ on $\R^n$ between vector spaces $V$ and $W$, and
Dirichlet boundary conditions, in the sense that admissible fields are of the
form $v=v_0+\varphi$, with $\mathscr{A}$-free $\varphi\in C_c^\infty(\Omega,V)$.
Our analysis also covers the ``potentials case'' $$ \mbox{minimize}\quad
\F(u)=\int_\Omega f(\mathscr{B} u(x))\mathrm{d} x\quad\text{for }
u\in u_0+ C_c^\infty(\Omega,U), $$ where $\mathscr{B}$ is another linear constant rank PDE operator on $\R^n$ between vector spaces $U,\,V$. We also prove appropriate higher integrability of
minimizers for both types of problems. In addition, our approach covers  non-autonomous integrands $f(x,v(x))$ or $f(x,\B u(x))$.
 \end{abstract}
\tableofcontents
\section{Introduction}

Let $\Omega\subset\R^n$ be an open bounded set and $1\leq p<\infty$. Consider the problem
\begin{align}\tag{$P\A_\Omega$}\label{eq:P_Omega_Afree}
\mbox{minimize}\quad \mathcal E(v,\Omega)=\int_\Omega f(x,v(x))\dif x\quad\text{for }v\in v_0+C_{c,\A}^\infty(\Omega),
\end{align}
where $v_0\in L^p(\Omega,V)$ is $\A$-free and
$
C_{c,\A}^\infty(\Omega)=\{\varphi\in C_c^\infty(\Omega,V)\colon \A \varphi=0\}
$. 

Here $\A$ is a homogeneous vectorial differential operator with constant coefficients, defined on $\R^n$ from $V$ to $W$, which are both finite dimensional inner product spaces. We denote its characteristic polynomial/symbol map by $\A(\xi)\in \lin(V,W)$ for $\xi\in\R^n$. For the precise definitions, see Section~\ref{sec:diff_op}.

Functionals of this type have been studied extensively over the last several decades, first in the prototypical case $\A=\curl$ (which leads to functionals defined on Sobolev spaces) and later in the case of general differential constraints, emanating from the study of compensated compactness \cite{Tartar78, Murat81}. The lower semicontinuity of these functionals with respect to weak convergence is well understood both in the $\curl$-free case \cite{Morrey52,Dacorogna82_JFA,AF84,FonMul93,AmDal92,KP94,kristensen1999lower,FMP98,KriRin10a} as well as in the general $\A$-free case \cite{Murat81,Dacorogna82,FonMul99,FLM04,Dacorogna_JEMS,ARDPR,AR21,GR22,KriRai22}.  The overarching theme in these works is that the lower semicontinuity of functionals in \eqref{eq:P_Omega_Afree} is equivalent to the so-called \textbf{$\A$-quasiconvexity} condition introduced in \cite{FonMul99}.

All these results are subject to the so called \textbf{constant rank} condition on $\A$, i.e.
\begin{align}\tag{CR}\label{eq:CR}
    \rank\A(\xi)=\text{const.}\quad\text{for }\xi\in\R^n\setminus\{0\}.
\end{align}
Except for notable examples \cite{Muller99,LMM,SzWi}, the lower semicontinuity problem is widely open in the absence of the constant rank condition.
We will assume \eqref{eq:CR}, together with the condition that the wave cone of $\A$ \cite{Tartar78,Murat81} is spanning, i.e.
\begin{align}\tag{SC}\label{eq:SC}
    \spann\bigcup_{\xi\in\R^n\setminus\{0\}}\ker\A(\xi)=V.
\end{align}
This second assumption is mild (see Section~4.1 in \cite{KriRai22}).

Under these assumptions, the regularity of minimizers is far from fully understood. Our aim is to reach the high benchmark established in the case $\A=\curl$ by works such as \cite{Evans86,AF87,AF89,GiaMod86,marcellini1989regularity,CFM98,diening2012partial,GmeKri_BV}.
 These works use stronger quasiconvexity conditions to infer partial regularity results, which is typical in the vectorial calculus of variations even for strictly convex problems. Strong quasiconvexity also leads to higher integrability of minimizers above the growth parameter, see \cite{GG82} for a classic reference. For more details see the comprehensive exposition in \cite{Min06}.

On the other hand, there are only a handful of contributions to regularity for the general Problem~\eqref{eq:P_Omega_Afree} \cite{Federico,GK19,Gme20,Gme21,CG22,BK22,Schiffer24}, all of which impose stringent restrictions on the setup. 

Our main result establishes partial continuity and higher integrability of minimizers of \eqref{eq:P_Omega_Afree} under a suitable assumption of strong quasiconvexity on the integrand $f$: 
\begin{bigthm}\label{thm:main_A-free}
Let $\A$ be a  differential operator satisfying the \eqref{eq:CR} and \eqref{eq:SC} conditions and let $1<p<\infty$. Suppose that $f\colon \Omega\times V\to\R$ satisfies the  growth, smoothness, and strong $\A$-quasiconvexity  {assumptions} \textup{\ref{it:fgrowth}}, \textup{\ref{it:f_Acontinuity}}$'$, and \textup{\ref{it:f_B-qcstrong}} below. Then for any generalized minimizer $v$ of \eqref{eq:P_Omega_Afree}, there exists an open subset $\Omega'\subset\Omega$ and a constant $\sigma_0>0$ such that $\mathscr{L}^n(\Omega\setminus \Omega^{\prime})=0$ and $v\in C^{0,\alpha}(\Omega',V)\cap L^{p+\sigma}_{\locc}(\Omega,V)$ for any $\alpha\in(0,1)$ and $\sigma\in(0,\sigma_0)$.
\end{bigthm} 
Here $\A$ is assumed to have homogeneous entries in the derivatives for simplicity of exposition, but the same result holds for differential operators with homogeneous rows (see \cite{Raita_pot_new} and Remark~\ref{rmk:homogeneity}). The case $p=1$ can be addressed by a non-trivial upgrade of our method, which will make the object of our future work.
We emphasize that Theorem~\ref{thm:main_A-free} is new for \textit{autonomous} integrands, i.e., when $f(x,\,\cdot\,)$ is independent of $x$; our method of proof integrates the $x$-dependence seamlessly.

Most of the earlier results concerning general differential operators that we mentioned pertain to the related problem on potential operator level
\begin{align}\tag{$P\B_\Omega$}\label{eq:P_Omega}
\mbox{minimize}\quad \F(u,\Omega)=\int_\Omega f(x,\B u(x))\dif x\quad\text{for }u\in u_0+C_c^\infty(\Omega, { U}),
\end{align}
where $u_0\in W^{k-1,p}(\Omega,U)$ is such that $ \B u_0\in L^p(\Omega,V)$; here $k$ is the order of $\B$.

We will show that establishing partial continuity of minimizers of \eqref{eq:P_Omega_Afree} can be reduced to proving partial continuity of minimizers $\B u$ of \eqref{eq:P_Omega} (see Remark~\ref{rmk:reduction}). Here $\A$ and $\B$ satisfy a homological relation described in Section~\ref{sec:diff_op} (see also \cite{Raita19,Raita_pot_new}). The reduction of $\eqref{eq:P_Omega_Afree}$ to $\eqref{eq:P_Omega}$ in the context of regularity theory is very likely to be new. 

Earlier results pertaining to partial regularity of minimizers of \eqref{eq:P_Omega} in \cite{CG22,Federico,} concern \textit{elliptic} operators $\B$ of first order. Ellipticity is a restrictive assumption which allows the use of strong harmonic analysis estimates.  We prove partial continuity and higher integrability of minimizers $\B u$ of \eqref{eq:P_Omega} for arbitrary constant rank operators:
\begin{bigthm}\label{thm:main}
   Let $\B$ be a  differential operator satisfying \eqref{eq:CR} and let $1<p<\infty$. Suppose that $f\colon \Omega\times V\to\R$ satisfies the  growth, smoothness, and strong quasiconvexity {assumptions} \textup{\ref{it:fgrowth}}, \textup{\ref{it:fcontinuity}}, and \textup{\ref{it:f_B-qcstrong}}$'$ below. Then for any generalized minimizer $\B u $ of \eqref{eq:P_Omega}, there exists an open subset $\Omega'\subset\Omega$ satisfying $\mathscr{L}^n(\Omega\setminus \Omega^{\prime})=0$ and a constant $\sigma_0>0$ such that $\B u\in C^{0,\alpha}(\Omega',V)\cap L^{p+\sigma}_{\locc}(\Omega,V)$ for any $\alpha\in(0,1)$, $\sigma\in(0,\sigma_0)$. Moreover, for each $\omega\Subset\Omega$, there is a representative $\tilde u\in W^{k,p}(\Omega,U)$ such that $\B u = \B \tilde{u}$ in $\omega$ and $\tilde u\in C^{k,\alpha}(\omega\cap\Omega',U)\cap W^{k,p+\sigma}(\omega,U)$. 
\end{bigthm}
As is the case for Theorem~\ref{thm:main_A-free}, the result of Theorem~\ref{thm:main} is new for autonomous integrands.

Unlike for elliptic operators $\B$, the potential $u$ associated with a minimizer $\B u$ can be wildly irregular since non-elliptic operators $\B$ have large kernels. In fact, it was shown in \cite{CG22} that ellipticity is necessary for minimizers $u$ of \eqref{eq:P_Omega} to be partially $C^{k,\alpha}$-regular. This degeneracy also makes the proof of existence of generalized minimizers of Problem~\eqref{eq:P_Omega}  more involved, see the details in Section~\ref{susbsec:B_exist}. 

The natural coercivity condition given by the appropriate strong quasiconvexity assumption becomes insufficient to choose a good representative (Remark~\ref{rmk:Bcoercivity}). This is also reflected in the Caccioppoli inequality in Proposition~\ref{prop:Caccioppoli} which we found insufficient for performing an excess-decay estimate approach. 
Our main new idea is to look at a potential operator $\C$ of $\B$ (Theorem~\ref{thm:BC}) and to attempt to impose the condition $\C^* u=0$ in \eqref{eq:P_Omega}. It is unclear if this is possible in general open sets, so we only worked with local  estimates, see Section~\ref{sec:linear_inequ}. One difficulty in proving these comes from the fact that the order of $\C$ is in general larger than the order of $\B$, so it was unclear what kind of Calder\'on--Zygmund estimates to expect.
The nonlinear estimate in \eqref{eq:Korn_ball} lead us to the improved Caccioppoli inequality in Proposition~\ref{prop:Caccioppoli_Dk}, which is strong enough to enable us to prove both the higher integrability in Section~\ref{sec:higher_int} and the partial regularity in Section~\ref{sec:partialreg}.

Both  our main results above also apply to local minimizers in the sense of 
\cite{CG22,Federico}, i.e. vector fields $v\in L^p(\Omega,V)$ such that $\A v=0$ which satisfy
$$
\int_\Omega f(x,v(x))\dif x\leq \int_\Omega f(x,v(x)+\varphi(x))\dif x \quad\varphi\in C_c^\infty(\Omega,V)\text{ with }\A\varphi=0.
$$
For the $\B$-gradients setting we consider $v=\B u$ and $\varphi=\B\phi$ with $\phi\in C_c^\infty(\Omega,U)$.

We remark that in contrast to Theorem~\ref{thm:main_A-free}, in Theorem~\ref{thm:main}, we require no spanning cone condition. This is so since in the $\B$-gradient framework, this condition is satisfied automatically, see Lemma~\ref{lem:essential_range}.
We also clarify that in both Theorem~\ref{thm:main_A-free} and~\ref{thm:main} the notion of generalized minimality is considered with respect to the natural sequential weakly $L^p$-closure of the admissible classes. The details are provided in Section~\ref{sec:existence}.

\par We now list our 
assumptions on  integrands $f\col \Omega \times V \to \R$: 
    \begin{enumerate}[label=(\textbf{H\theenumi}),leftmargin=3\parindent, itemsep=.2em]
        \item\label{it:fgrowth} $\abss{f(x,z)}\leq L(1+\abss{z}^p)$ for any $(x,z)\in \Omega \times V$, where $L>0$;
        \item\label{it:fcontinuity} For any $x\in \Omega$ the function $f(x,\cdot)$ is $C^{2}$, and for any $z\in V$ the functions $\partial_z^{\, i}f(\cdot, z)$, $i=0,1,2$, are all continuous with
            \begin{equation*}
                \abss{\partial_z f(x,z)-\partial_z f(y,z)}\leq L \abss{x-y}(1+\abss{z}^2)^{\frac{p-1}{2}}
            \end{equation*} for any $z\in V$ and any $x,y\in \Omega$; 
        \item[(\textbf{H2})$^{\prime}$] The functions $\partial_z^{\,i}f(\cdot, \cdot)$, $i=0,1,2$, are continuous in $\Omega \times V$ with \label{it:f_Acontinuity}
            \begin{align*}
                \abss{\partial_z f(x,z)-\partial_z f(y,z)}&\leq L \abss{x-y}(1+\abss{z}^2)^{\frac{p-1}{2}},\\
                \abss{\partial_z^2 f(x,z)} &\leq L (1+\abss{z}^2)^{\frac{p-2}{2}}
            \end{align*} for any $z\in V$ and any $x,y\in \Omega$;
        \item The function $f$ is \textbf{strongly $\A$-quasiconvex} in the following sense:  there exists $\ell>0$ such that for any $z\in V$,  $\varphi \in C_c^{\infty}(B_1,V)$ with $\A \varphi =0$ and a.e. $x_0\in\Omega$, we have \label{it:fqcstrong}
            \begin{equation*}
                \fint_{B_1} f(x_0,z+ \varphi) \dif x \geq f(x_0,z) + \ell \fint_{B_1} (1+\abss{z}^2 +\abss{\varphi}^2)^{\frac{p-2}{2}}\abss{\varphi}^2 \dif x.
            \end{equation*} 
        \item[(\textbf{H3})$^{\prime}$] 
            The function $f$ is \textbf{strongly quasiconvex with respect to $\B$-gradients} in the following sense: there exists $\ell>0$ such that for any $z\in V$, $\phi \in C_c^{\infty}(B_1,U)$ and a.e. $x_0\in\Omega$, we have\label{it:f_B-qcstrong}
            \begin{equation*}
                \fint_{B_1} f(x_0,z+ \B\phi) \dif x \geq f(x_0,z) + \ell \fint_{B_1} (1+\abss{z}^2 +\abss{\B\phi}^2)^{\frac{p-2}{2}}\abss{\B\phi}^2 \dif x.
            \end{equation*} 
    \end{enumerate}
\par Assumption \ref{it:f_Acontinuity}$'$, which is slightly stronger than \ref{it:fcontinuity}, is needed for the regularity claim in Theorem \ref{thm:main_A-free}. We will locally reduce the $\A$-free problem \eqref{eq:P_Omega_Afree} to \eqref{eq:P_Omega} for the regularity result (see Section \ref{sec:existence}), and \ref{it:f_Acontinuity}$'$ guarantees that the new integrand obtained in the reduction satisfies \ref{it:fcontinuity}.
\par Notice that \ref{it:fgrowth} and $\ref{it:fqcstrong}$ together imply, by Lemma 2.3 in \cite{KK16},
    \begin{equation}\label{eq:fLip}
        \abss{\partial_z f(x,z)} \leq CL(1+\abss{z}^{p-1})\quad\text{for any $(x,z) \in \Omega \times V$,}
    \end{equation} 
    an estimate which will be used frequently in the sequel. 

It was observed in \cite{ChenKris17} that strong quasiconvexity at a point is sufficient for proving mean coercivity for autonomous integrands. This interesting observation led us to realize that our strong quasiconvexity assumptions ensure existence in the case of autonomous integrands:
\begin{bigthm}\label{thm:existence_autonomous}
    Let $f\colon V\to\R$. Under the assumptions of Theorem~\ref{thm:main_A-free} (resp.~\ref{thm:main}), generalized minimizers of Problem~\eqref{eq:P_Omega_Afree} (resp.~\eqref{eq:P_Omega}) exist.
\end{bigthm}
The technical description of our notion of minimality can be found in Definition~\ref{def:generalized_min}. As a consequence of Theorem~\ref{thm:existence_autonomous}, our results offer a full package in the case of autonomous integrands: strong quasiconvexity, smoothness, and growth conditions imply existence and regularity of minimizers.

When applied to the case of exterior derivatives, this result is  more general than Theorem~5.1 in \cite{Dacorogna_JEMS} as far as both the coercivity condition and the regularity of the boundary data are concerned. In fact, to the best of our knowledge, our results are the first partial regularity results for minimizers of integral functionals defined on exact differential forms \cite{Dacorogna_JEMS,SIL1,SIL2,SIL3}:
\begin{bigthm}\label{thm:main_du}
Let $1<p<\infty$, $u_0\in L^p(\Omega,\Lambda^{k-1} \R)$ with $du_0\in L^p(\Omega)$. Consider the problem
\begin{align}\label{eq:P_Omega_du}
\mbox{minimize}\quad \int_\Omega f(du)\quad\text{for }u\in u_0+C_c^\infty(\Omega, \Lambda^{k-1}\R^n),
\end{align}
where $f\colon \Lambda^k\R^n\to\R$ satisfies the  growth, smoothness, and strong quasiconvexity {assumptions} \textup{\ref{it:fgrowth}}, \textup{\ref{it:fcontinuity}}, and \textup{\ref{it:f_B-qcstrong}}$'$ with $\B=d$. Then  generalized minimizers $du$ of \eqref{eq:P_Omega_du} exist. For any such minimizer, there exists an open subset $\Omega'\subset\Omega$ satisfying $\mathscr{L}^n(\Omega\setminus \Omega^{\prime})=0$ and a constant $\sigma_0>0$ such that $d u\in C^{0,\alpha}(\Omega',\Lambda^{k}\R^n)\cap L^{p+\sigma}_{\locc}(\Omega,\Lambda^{k}\R^n)$ for any $\alpha\in(0,1)$ and $\sigma\in(0,\sigma_0)$. Moreover, for each $\omega\Subset\Omega$, there is a representative $\tilde u\in W^{1,p}(\Omega,\Lambda^{k-1}\R^n)$ such that $d u = d \tilde{u}$ in $\omega$ and $\tilde u\in C^{1,\alpha}(\omega\cap\Omega',\Lambda^{k-1}\R^n)\cap W^{1,p+\sigma}(\omega,\Lambda^{k-1}\R^n)$. 
\end{bigthm}
Despite precise Hodge decompositions on domains in this case \cite{ISS,schwarz2006hodge}, we cannot obtain the existence of a representative $\tilde u$ such that $d\tilde u=d u$ in $\Omega$ due to the low boundary regularity of $\Omega$ (open set). This would be possible  if $\Omega$ is assumed Lipschitz for instance.

Our last main result concerns the partial continuity and higher integrability for functionals defined on closed differential forms:
\begin{bigthm}\label{thm:main_d-free}
    Let $1<p<\infty$ and $v_0\in L^p(\Omega,\Lambda^k \R^n)$ satisfy $dv_0=0$. Suppose that $f\colon  \Lambda^k\R^n\to\R$ satisfies the  growth, smoothness, and strong $d$-quasiconvexity  {assumptions} \textup{\ref{it:fgrowth}}, \textup{\ref{it:fcontinuity}}$'$, and \textup{\ref{it:f_B-qcstrong}} for $\A=d$. Then there exist generalized minimizers $v$ of the problem 
    \begin{align*}
\mbox{minimize}\quad \int_\Omega f(v)\quad\text{for }v\in v_0+C_{c,d}^\infty(\Omega).
\end{align*}
   For any such minimizer $v$, there exists an open set $\Omega'\subset\Omega$ and a constant $\sigma_0>0$ such that $\mathscr{L}^n(\Omega\setminus \Omega^{\prime})=0$ and $v\in C^{0,\alpha}(\Omega',\Lambda^{k}\R^n)\cap L^{p+\sigma}_{\locc}(\Omega,\Lambda^{k}\R^n)$ for any $\alpha\in(0,1)$ and $\sigma\in(0,\sigma_0)$.
\end{bigthm}
This result is an immediate application of Theorem~\ref{thm:main_A-free}. Both Theorems~\ref{thm:main_du} and~\ref{thm:main_d-free} extend to the case of vector valued differential forms.  While there has been work on regularity for functionals and nonlinear elliptic systems defined on exact or closed differential forms \cite{uhlenbeck1977regularity,hamburger1992regularity,beck2013regularity,lee2024calderon}, we believe that Theorems~\ref{thm:main_du} and~\ref{thm:main_d-free} are new. In fact, it may well be that Theorem~\ref{thm:main_A-free} is new for $\A=\di$, $\A\colon \mathscr D'(\R^n,\R^{N\times n})\to\mathscr D'(\R^n,\R^{n})$.
\begin{example}
    In the following, we list a few more examples of PDE constraints $\A$ of constant rank appearing in  variational models for various applications. To the best of our knowledge, the regularity results established in Theorem~\ref{thm:main_A-free} are new for all the following examples:
    \begin{enumerate}
        \item \textit{Divergence on symmetric matrices}: $\A\colon\mathscr D'(\R^n,\R^{n\times n}_{\sym})\to\mathscr D'(\R^n,\R^{n})$ is the row-wise divergence operator. Symmetric div-quasiconvex hulls were studied in connection with elasticity theory in \cite{conti2018data,conti2020symmetric,behn2022symmetric}.
        \item \textit{The linearly relaxed Euler system}, $\A\colon\mathscr D'(\R^{1+d},\R\times\R^d\times \mathbb S^d_0\times \R)\to\mathscr D'(\R^{1+d},\R^{1+d})$,      
        $$\A_E(\rho,m,M,q)=(\partial_tm+\di M+Dq,\,\partial_t\rho+\di m),$$ 
        which was used for selecting measure-valued solution of the isentropic Euler system \cite{gallenmuller2021muller,gallenmuller2021selection,gallenmuller2023measure,gallenmuller2023which,gallenmuller2023probabilistic}, see also \cite{delellis2009euler,SzWi}.
        \item \textit{Generalized Saint-Venant compatibility operators}: The higher order operators defined in Section~2.2, Equation~(14) of ~\cite{raictua2024scaling} are used in inverse problems and tensor tomography \cite{sharafutdinov2012integral,paternain2014tensor,paternain2023geometric}.
        \item \textit{Higher order divergence}: The operator $\A=\di^k=(D^k)^*$, i.e. the adjoint of the $k$-th order gradient gradient, $\A\colon \mathscr D'(\R^n,\mathrm{Slin}^k(\R^n,\R^N))\to\mathscr D'(\R^n,\R^{N})$. This is an ubiquitous  example since any operator of order $k$ can be written as $\di^k\circ T$ for some linear map $T$. This operator has several invariance properties which were used to prove very general statements, e.g. \cite{van2008estimates,van2013limiting,van2022injective}.    
        \end{enumerate}
\end{example}
Our work is just an intersection node between two large networks of results in the calculus of variations: regularity theory and variational problems under linear pde constraints. While no amount of citations would do justice to either field, we will nevertheless endeavor to point the reader in a few fascinating directions that are related to our theme.
We begin by mentioning regularity results for the widely studied $\curl$-free case, which is Problem~\eqref{eq:P_Omega_Afree} with $\A=\curl$ or Problem~\eqref{eq:P_Omega} for $\B=D$, with standard growth conditions including linear growth \cite{carozza2003partial,kristensen2003partial,kuusi2016partial,li2022partial,li2023regularity} and  non-standard growth conditions \cite{marcellini1989regularity,fonseca1997relaxation,esposito1999higher,esposito2004sharp,bella2020regularity,schmidt2009regularity,Schmidt14,carozza2021trace,defilippis2022quasiconvexity,irving2023bmo,defilippis2023singular,diening2012partial}. Problems of the type \eqref{eq:P_Omega_Afree} with the same Dirichlet boundary conditions were considered recently in \cite{raictua2023scaling,raictua2024scaling}, see also \cite{garroni2004rigidity,palombaro2004three,chan2015energy,dephilippis2018two,sorella2023four,ruland2022energy}. Contributions related to concentration phenomena in $\A$-free sequences include \cite{fonseca2010oscillations,KK16,dephilippis2016structure,dephilippis2017characterization,arroyo2019dimensional,guerra2022compensated,KriRai22,derosa2023fine,arroyo2024higher,guerra2024compensation}. Symmetric div-quasiconvex hulls of sets, with an outlook to geometrically linear elasticity, were studied in \cite{conti2018data,conti2020symmetric,behn2022symmetric}. Interesting  contributions to the study of the Euler equations in the $\A$-free frameworks were made in \cite{chiodaroli2017free,gallenmuller2021selection,gallenmuller2023measure,gallenmuller2023which,gallenmuller2023probabilistic}. Homogenization results for $\A$-quasiconvex integrals can be found in \cite{braides2000quasiconvexity,matias2015homogenization,davoli2016homogenization,davoli2016periodic,davoli2018relaxation,ferreira2021homogenization,davoli2024homogenization}. Applications to image processing can be found in \cite{pagliari2022bilevel,davoli2023adaptive}.
Results on lower semicontinuity and regularity of minimizers of integrals defined on spaces of mixed smoothness can be found in \cite{kazaniecki2017anisotropic,prosinski2019calculus,prosinski2023existence}.
Other interesting contributions that do not fit in any of the above categories can be found in \cite{dacorogna2002ab,santos2004quasi,baia2013lower,kramer2017A,prosinski2018closed,koumatos2021quasiconvexity,skipper2021lower,lienstromberg2023data,schiffer2024sufficient}.

\medskip
This paper is organized as follows: In Section~\ref{sec:prel} we collect notation and preliminaries on differential operators and function space inequalities, in Section~\ref{sec:coercivity} we prove mean coercivity for autonomous integrands and lower semicontinuity for non-autonomous integrands, paving the way to the existence proofs in Section~\ref{sec:existence}, where we also prove Theorem~\ref{thm:existence_autonomous}. In Section~\ref{sec:linear_inequ} we prove local estimates for linear systems which are crucial for our main results. In Section~\ref{sec:higher_int} we prove two Caccioppoli inequalities and show how they lead to higher integrability. Finally, in Section~\ref{sec:partialreg} we prove the excess decay estimates that lead to partial regularity and complete the proofs of Theorems~\ref{thm:main_A-free} and~\ref{thm:main}.

\subsection*{Acknowledgment} The authors thank Tatiana Toro for insightful discussions. B.R. thanks SLMath and Tatiana Toro for the hospitality and financial support during a visit when significant parts of the current research were conducted. Part of this paper is based upon work supported by the National Science Foundation under Grant No. DMS-1928930 and by the Alfred P. Sloan Foundation under grant G-2021-16778, while Z.L. was in residence at the Simons Laufer Mathematical Sciences Institute (formerly MSRI) in Berkeley, California, during the Spring 2024 semester.

\section{Preliminaries}\label{sec:prel}
\subsection{Notation}\label{subsec:notation}
\par Throughout this paper, $\Omega$ is a bounded open subset of $\R^n$ unless otherwise specified. For any $x\in \R^n$ and $r>0$, we denote by $B(x,r)$ or $B_r(x)$ the open ball $\{y\in \R^n\col \abss{x-y}<r\}$. $\Le^n$ is the Lebesgue measure on $\R^n$, and $\indi_{E}$ for any $E\subset \R^n$ is the indicator function of $E$. For any measurable set $E\subset \R^n$ with $0<\Le^n(E)<\infty$ and any $f \in L^1(E)$, set the average notation as follows
    \[ (f)_E= \fint_{E}:= \frac{1}{\Le^n(E)}\int_E f\dif x. \] This notation also extends to vector-valued functions.
\par For any exponent $p \in (1, \infty)$, its conjugate $p'$ is defined such that $\frac{1}{p}+\frac{1}{p'}=1$. The comparability of two quantities $x_1$ and $x_2$ is denoted by $x_1 \sim_{c_1,\ldots,c_m} x_2$, which means that there exists $C=C(c_1,\ldots,c_m)\geq 1$ such that $\frac{1}{C}x_2 \leq x_1 \leq Cx_2$. Similarly, the relation $x_1 \lesssim_{c_1,\ldots,c_m} x_2$ means that there exists $C=C(c_1,\dots,c_m)>0$ such that $ x_1 \leq Cx_2$.

\par Suppose that $X$ is an arbitrary finite dimensional inner-product space over $\R$. For any $x,y \in X$, denote by $x\cdot y$ or $\brangle{x,y}$ the inner product of $x$ and $y$. The norm induced by the inner product is denoted by $\abss{\,\cdot\,}$ (i.e., $\abss{x} = (x\cdot x)^{1/2}$). The bracket function $\brangle{\cdot}$ is defined by $\brangle{x}:= (1+\abss{x}^2)^{1/2}$. Given any subspace $X_1 \subset X$, its orthogonal complement is 
    \[ X_1^{\perp}:= \{y \in X\col x\cdot y = 0 \ \mbox{for any }x\in X_1\}. \] 
\par For any open subset $O \subset \R^n$, $C_c^{\infty}(O,X)$ is the space of $X$-valued test maps on $O$, namely the $C^{\infty}$ maps $\varphi\col O \to X$ with compact supports contained in $O$. Correspondingly, the space of $X$-valued distributions on $O$ is denoted by $\mathscr{D}'(O,X)$. We denote by $\mathscr S(\R^n,X)$ the Schwarz class of rapidly decreasing functions and by $\mathscr S'(\R^n,X)$ its dual space of tempered distributions.

\par Given two finite dimensional inner-product spaces $X,Y$ as above, the space of $k$-linear maps from $X$ to $Y$ is denoted by $\lin^k(X,Y)$, abbreviated $\lin(X,Y)$ if $k=1$. These maps can be identified with $(k+1)$-tensors $(c_{i_0i_1\ldots i_k})$ by
$$
\langle T(x^1,\ldots, x^k),y\rangle=\sum_{i_0,i_1,\ldots i_k} c_{i_0i_1\ldots i_k} y_{i_0}x^1_{i_0}\ldots x_{i_j}^k
,\quad\text{for }y\in Y,x^j\in X,j=0,\ldots, k,$$
where $i_0$ runs over an orthonormal basis of $Y$ and the other indices count an orthonormal basis of $X$. We will work with the $L^2$-norm of $T$, namely
$$
|T|^2=\sum_{i_0,i_1,\ldots i_k} c_{i_0i_1\ldots i_k}^2.
$$
The space $\mathrm{Slin}^k(X,Y)$ of symmetric $k$-linear maps from $X$ to $Y$ is defined by requiring
$$
c_{i_0\sigma(i_1)\ldots \sigma(i_k)}=c_{i_0i_1\ldots i_k}\quad\text{for any permutation $\sigma$ of $\{1,2,\ldots,k\}$.}
$$
Equivalently, $T(x^{\sigma(1)},\ldots x^{\sigma (k)})=T(x^1,\ldots,x^k)$.
Importantly for us, if $u\in C^\infty(X,Y)$, we have that $D^ku\in \mathrm{Slin}^k(X,Y)$. Moreover, 
$$
y\otimes x^{\otimes k}:=(y_{i_0}x_{i_1}x_{i_2}\ldots x_{i_k})_{i_0,i_1,\ldots i_k}\in \mathrm{SLin}^k(X,Y).
$$
The higher order gradients and exterior products above are connected by Fourier transform: Suppose that $f \in L^1(\R^n, Y)$ with $Y$ being a finite dimensional inner-product space as above. Then the Fourier transform of $f$ is
    \[ \FT f(\xi) = \hat{f}(\xi):= \int_{\R^n} f(x)e^{-\mathrm{i}2\pi  x\cdot \xi}\dif x, \quad \xi \in \R^n. \]
    If $u\in \mathscr S(\R^n,Y)$, we have that $\widehat{D^ku}(\xi)=c\hat u(\xi)\otimes \xi^{\otimes k}$ for $\xi\in\R^n$ for a complex constant $c$. 

One can check that $\abss{y\otimes x^{\otimes k}}=|y||x|^k$ where $|\cdot|$ denotes $L^2$-norms on $X$ and $Y$ (these can be considered with respect to arbitrary orthonormal bases).

For any $A \in \lin(X,Y)$, define its image and kernel as
    \[ \rmim A:= \{y \in Y\col y = A x \ \mbox{for some }x\in X\}, \quad \ker A:= \{x \in X\col A x =0\in Y\}.\] The {\it Moore Penrose inverse} $A^{\dagger}$ of $A$ is given by 
    \[ A^{\dagger}:= \left\{
    \begin{aligned}
        &(A\vert_{(\ker A)^{\perp}})^{-1},& &\mbox{on }\rmim A,\\
        &0,& &\mbox{on }(\rmim A)^{\perp}, 
    \end{aligned}\right. \] with which we can write the orthogonal projections onto $\rmim A$ and $\rmim A^{\ast}$ as follows:
    \begin{equation} \label{eq:operator_proj}
    \proj_{\rmim A}:= A A^{\dagger}, \quad \proj_{\rmim A^{\ast}}:= A^{\dagger}A. \end{equation} It is easy to check that the following algebraic relations hold true:
    \begin{equation}\label{eq:operator_alg}
    AA^{\dagger}A = A,\quad A^{\dagger}AA^{\dagger} = A^{\dagger},\quad (AA^{\dagger})^{\ast}= AA^{\dagger}, \quad (A^{\dagger}A)^{\ast}=A^{\dagger}A.
    \end{equation} In fact each of \eqref{eq:operator_proj} and \eqref{eq:operator_alg} give equivalent definitions of $A^{\dagger}$.
 
\subsection{Differential operators}\label{sec:diff_op}
 We will work with differential operators
\begin{align*}
\A=\sum_{|\alpha|=h}A_\alpha\partial^\alpha,\quad\B=\sum_{|\beta|=k}B_\beta\partial^\beta,\quad\C=\sum_{|\gamma|=l}C_\gamma\partial^\gamma,
\end{align*} 
where
\[A_{\alpha} \in \lin(V,W), \quad B_\beta\in \lin(U,V), \quad C_{\gamma}\in \lin(Z,U) \] are matrix coefficients and $Z,U,V,W$ are finite dimensional inner-product spaces. We will see in the following that $Z$ is taken to be $U$ throughout the paper.
The three operators will be assumed to satisfy certain exact relations in Fourier variables that we now explain.
To this end, we will  use  the notation 
\begin{align*}
    \A(\xi)=\sum_{|\alpha|=h}A_\alpha\xi^\alpha,\quad\B(\xi)=\sum_{|\beta|=k}B_\beta\xi^\beta,\quad\C(\xi)=\sum_{|\gamma|=l}C_\gamma\xi^\gamma,
\end{align*}
for the characteristic (matrix) polynomials of respectively $\A$, $\B$, and $\C$. Any differential operator $\A$ as above can be written as $\A v=\di^h L v$ where $L\in \lin(V,\mathrm{SLin}^h(\R^n,W))$. To see this, let $v\in \mathscr D'(\R^n,V)$ and then we have 
\begin{align*}
&\A v=\sum_{|\alpha|=h}A_\alpha \partial^\alpha v=\sum_{|\alpha|=h} \partial^\alpha (A_\alpha v) =\di^h Lv,
\mbox{ where } Lz=(A_\alpha z)_{|\alpha|=h}\text{ for }z\in V.
\end{align*}
For completeness, we mention that $\di^h$ is the formal adjoint operator of $D^h$.

Recall that an operator $\B$ is said to be of \textbf{constant rank} if $\rank\B(\xi)=\text{const.}$ for all $\xi\in\R^n\setminus\{0\}$. A special important subclass is that of \textbf{elliptic} operators, by which we mean $\ker \B(\xi)=\{0\}$ for $\xi\in\R^n\setminus\{0\}$.
We will need the following result from \cite{Raita19}:
\begin{theorem}\label{thm:BC}    
Let $\B$ be an operator as above which is assumed to have constant rank. Then there exists an operator $\C$ as above, also of constant rank, such that
\begin{align}\label{eq:exact_B}
\rmim \C(\xi)=\ker\B(\xi)\quad\text{for all }\xi\in\R^n\setminus\{0\}.
\end{align}
Moreover, if $B_\beta\in \lin(U,V)$, we can choose $\C$ such that $C_\gamma\in\lin (U,U)$ and $l>k$.
\end{theorem}
The reason why we can define $\C$ on $U$ as well is by construction and we can choose $l>k$ because we can simply replace $\C$ with $\Delta^m\C$ for $m$ large enough.
By duality and iteration we can obtain the following:
\begin{corollary}\label{cor:ABC}
Let $\A$ be an operator as above which is assumed to have constant rank. Then there exist  operators $\B,\,\C$ as above, also of constant rank, such that
\begin{align}\label{eq:exact_A}
\rmim \B(\xi)=\ker\A(\xi),\quad\rmim \C(\xi)=\ker\B(\xi)\quad\text{for all }\xi\in\R^n\setminus\{0\}.
\end{align}
Moreover, if $A_\alpha\in \lin(V,W)$, we can choose $\B,\,\C$ such that $B_\beta,\,C_\gamma\in\lin (V,V)$ and $l>k>h$.    
\end{corollary}
\begin{remo}\label{rmk:homogeneity}    
It was observed in \cite{Raita_pot_new} that $\A$ can have homogeneous rows (i.e. if we write $\A(\xi)=(\A_{ij}(\xi))$, then $\{\A_{ij}\}_j$ is $h_i$-homogeneous) and the conclusion of Corollary~\ref{cor:ABC} remains identical. It can be showed that in these circumstances the conclusion of the main result Theorem~\ref{thm:main_A-free} remains the same. We do not include the details in the present work.
\end{remo}
We also clarify that in the case of Corollary~\ref{cor:ABC}, we have that $U=V$. With this assumption, we have that in both Theorem~\ref{thm:BC} and Corollary~\ref{cor:ABC}, we have $\A(\xi)\in\lin(V,W)$, $\B(\xi)\in\lin(U,V)$, $\C(\xi)\in\lin(U,U)$ for $\xi\in\R^n$.

We will also use the notation $\B u=T(D^ku)$, for a linear map $T\in\lin(\mathrm{SLin}^k(\R^n,U),V)$, which only amounts to collecting all the coefficients of $\B$ in a tensor. We also record that $\B(\xi)u_0=T(u_0\otimes \xi^{\otimes k})$.
With this notation we have the following observation:
\begin{lemma}\label{lem:essential_range}
    For any homogeneous linear differential operator $\B$ with constant coefficients we have that
    $$
    \spann\bigcup_{\xi\in\R^n\setminus\{0\}}\rmim \B(\xi)=\{\B u(x)\colon u\in C_c^\infty(\R^n,U),x\in\R^n\}=\rmim T.
    $$
\end{lemma}
\begin{proof}
First consider the case when $\B =D^k$, so $V=\mathrm{SLin}^k(\R^n,U)$ and $T=\id$. We then have
$$
    \spann\bigcup_{\xi\in\R^n\setminus\{0\}}U\otimes\xi^{\otimes k}=\{D^k u(x)\colon u\in C_c^\infty(\R^n,U),x\in\R^n\}=\mathrm{SLin}^k(\R^n,U), 
    $$
    which settles this case. To prove the general case, apply $T$ to the last equality to get
    $$
    \spann\bigcup_{\xi\in\R^n\setminus\{0\}}T(U\otimes\xi^{\otimes k})=\{T(D^k u(x))\colon u\in C_c^\infty(\R^n,U),x\in\R^n\}=T(\mathrm{SLin}^k(\R^n,U)), 
    $$
    which concludes the proof.
\end{proof}
\begin{remo}
    The relevance of Lemma~\ref{lem:essential_range} is that in the $\B$-gradient setting of Theorem~\ref{thm:main}, we can simply assume that $f\colon \Omega\times \rmim T\to\R$. The spanning cone condition is satisfied automatically. In fact, the spanning cone condition \eqref{eq:SC} is not at all restrictive even in the $\A$-free setting, as can be seen from {Lemma}~4.2 in \cite{KriRai22}.
\end{remo}

\subsection{Function spaces}\label{subsec:space}

\par Let $\Omega\subset\R^n$ be an open set. We define $L^{p}_{0,\A}(\Omega)$ as the strong $L^p$-closure of the set  $C_{c,\A}^\infty(\Omega)= \{\varphi \in C_c^{\infty}(\Omega,V)\colon \A \varphi =0\}$. This is a Banach space, as it is a closed subspace of $L^p(\Omega,V)$. We also define $X^{\B,p}_0(\Omega)$ as the $L^p$-closure of the set $\B C_c^\infty(\Omega,U)= \{\B \phi: \phi \in C_c^{\infty}(\Omega,U)\}$, which is also a Banach space. The target spaces of these operator-related function spaces are implicit given by the corresponding operators and thus omitted in the notation.

Of particular interest to the minimization problem \eqref{eq:P_Omega} are the spaces $W^{\B,p}(\Omega,U)=\{u\in W^{k-1,p}(\Omega,U)\colon \B u\in L^p(\Omega,V)\}$, and their relatives defined for $u\in\mathscr D'(\Omega,U)$ or $u\in L^p(\Omega,U)$. These spaces were studied extensively e.g. in \cite{RaitaTAMS,raita2020continuity,gmeineder2019embeddings,gmeineder2021limiting,gmeineder2024boundary,gmeineder2025korn,diening2020continuity,diening2025sharp,breit2020trace,hernandez2023endpoint,raictua2023trace}. We will not use them explicitly.

Given $1<p<\infty$, a weight $\ww\col \R^n \to [0,\infty)$ is said to be an $A_p$-weight (with notation $\ww \in A_p$) if $\ww \in L^1_{\locc}(\R^n)$, $0<\ww<\infty$ $\Le^n$-a.e. in $\R^n$, and there exists a constant $0< C<\infty$ such that
    \begin{equation}\label{eq:A_p}
    \brac{\fint_B \ww\dif x} \brac{\fint_{B}\ww^{-1/(p-1)}\dif x}^{p-1}\leq C 
    \end{equation} holds true for any ball $B$ in $\R^n$. The $A_p$-constant $[\ww]_{A_p}$ of $\ww$ is defined to be the least $C$ such that \eqref{eq:A_p} holds. Notice that if $\ww \in A_p$, we have $\ww' := \ww^{-1/(p-1)} \in A_{p'}$ with $[\ww']_{A_{p'}} = [\ww]_{A_p}^{p-1}$.
    
Let $m\geq 0$ be an integer. Given any finite dimensional inner-product space $X$, we define $\dot{W}{^{m,p}}(\Omega,X;\mathbf{w})$ as the closure of $C_c^\infty(\Omega,X)$ in the (semi-)norm $u\mapsto\|D^ku\|_{L^p(\Omega;\mathbf{w})}$. Its dual space is $\dot{W}{^{m,p}}(\Omega,X;\mathbf{w})^* =:\dot{W}{^{-m,p'}}(\Omega,X;\mathbf{w}')$. 
 
If $\Omega=\R^n$, we have the simple Fourier description:
    \begin{theorem}\label{thm:HMmultiplier}
    Let  $m$ be an integer, $1<p<\infty$, and $\mathbf w\in A_p$. Then
        \begin{align*}
          \normm{\varphi}_{\dot{W}^{m,p}(\R^n;\mathbf{w})}\sim   \normm{\FT^{-1} (\abss{\xi}^{m}\hat{\varphi})}_{L^p(\R^n;\mathbf{w})}\quad \mbox{for any }\varphi \in \dot{W}^{m,p}(\R^n,X;\mathbf{w}),
        \end{align*}
        with implicit constants depending on $n,p,m$ and $[\ww]_{A_p}$.
    \end{theorem} 
   \begin{proof}
       We first note that for $m=0$ there is nothing to prove. We then assume $m>0$ and calculate in Fourier space
       \begin{align*}
           \widehat{D^m\varphi}(\xi)=\hat\varphi(\xi)\otimes\xi^{\otimes m}=|\xi|^m\hat\varphi(\xi)\otimes\left(\frac{\xi}{|\xi|}\right)^{\otimes m},
       \end{align*}
       so by the H\"ormander--Mikhlin multiplier theorem and boundedness of Calder\'on--Zygmund operators with respect to $A_p$-weights, we obtain
       \begin{align}\label{eq:norm_equivalence}
           \normm{\varphi}_{\dot{W}^{m,p}(\R^n;\mathbf{w})}\leq C  \normm{\FT^{-1} (\abss{\xi}^{m}\hat{\varphi})}_{L^p(\R^n;\mathbf{w})}\quad \mbox{for any }\varphi \in \dot{W}^{m,p}(\R^n,X;\mathbf{w})
       \end{align}
       with $C=C(n,p,m,[\ww]_{A_p})$. The reverse inequality is established by using the algebraic identity
       \begin{align*}
           |\xi|^m\varphi(\xi)=\dfrac{\langle \xi^{\otimes m},\xi^{\otimes m}\rangle}{|\xi|^m}\varphi(\xi)=\left\langle\left(\frac{\xi}{|\xi|}\right)^{\otimes m},  \widehat{D^m\varphi}(\xi)\right\rangle
       \end{align*} and the same multiplier result as used in \eqref{eq:norm_equivalence}.
       It remains to establish the estimates for $m<0$, which we prove by duality. Let $f\in C_c^\infty(\R^n,X)$ and write 
       \begin{align*}
           \langle \varphi,f \rangle&=\langle \hat\varphi,\hat f \rangle =\langle |\xi|^{m}\hat\varphi,|\xi|^{-m}\hat f \rangle  =\langle \FT^{-1}(|\xi|^{m}\hat\varphi),\FT^{-1}(|\xi|^{-m}\hat f) \rangle  \\
           &\leq C \|\FT^{-1}(|\xi|^{m}\hat\varphi)\|_{L^p(\R^n;\mathbf{w})}\|f\|_{\dot W{^{-m,p'}}(\R^n;\mathbf{w}')},
       \end{align*}
       where in the inequality we used the Fourier characterization for $-m>0$. This implies \eqref{eq:norm_equivalence} also for $m<0$.

       We have one more inequality left to prove, for which we test $\varphi$ with $g\in \dot{W}^{-m,p'}(\R^n,X;\ww')$ and
       \begin{align*}
           \left\langle \FT^{-1}(|\xi|^m\hat\varphi),g \right\rangle=\left\langle |\xi|^m\hat\varphi,\hat g \right\rangle=\left\langle \hat\varphi,|\xi|^m\hat g \right\rangle\leq \|\varphi\|_{\dot W{^{m,p}(\R^n;\mathbf{w})}} \|\FT^{-1}(|\xi|^m\hat g)\|_{\dot W{^{-m,p'}(\R^n;\mathbf{w}')}}.
       \end{align*}
       Using the Fourier characterization for $-m>0$ we obtain
       \begin{align*}
           \|\FT^{-1}(|\xi|^m\hat g)\|_{\dot W{^{-m,p'}(\R^n;\mathbf{w}')}}\leq C\|\FT^{-1}(|\xi|^{-m}|\xi|^m\hat g)\|_{L{^{p'}(\R^n;\mathbf{w}')}}=C\|g\|_{L^{p'}(\R^n;\mathbf{w}')}.
       \end{align*}
       The latest two displayed formulas imply the remaining inequality.
       \end{proof}

\par In Section \ref{sec:linear_inequ}, we use weighted estimates and extrapolation to obtain a modular estimate \eqref{eq:Korn_ball}, which plays a role in the proof of our regularity result. For that purpose, we introduce the concept of Young functions: a function $\varphi \col [0,\infty) \to [0,\infty)$ is a Young function if it is continuous, strictly increasing, convex, and satisfies
    \[ \varphi(0)=0, \quad \lim_{t\to 0} \frac{\varphi(t)}{t} = 0, \quad \lim_{t\to \infty} \frac{\varphi(t)}{t}=\infty. \] A Young function is said to satisfy the $\Delta_2$-condition if there exists $D \geq 1$ such that $\varphi (2t) \leq D\varphi(t)$ for any $t >0$, and denote by $\Delta_2(\varphi)$ the least $D$ such that the above inequality holds true. The conjugate function $\varphi^{\ast}$ of a Young function $\varphi$ is defined as $\varphi^{\ast}(t):= \sup_{s\geq 0}(st -\varphi(s))$, and we say that $\varphi$ satisfies the $\nabla_2$-condition if $\varphi^{\ast}$ satisfies the $\Delta_2$-condition.
\par The following extrapolation result is needed as well. It is proved in \cite[Theorem 3.1]{CMP06}, see also \cite[Theorem 4.15]{Cruz-UribeEtAl2011}. The following version is closer to Proposition 6.1 in \cite{Diening_extra10}, and such a result is also used in a similar context in \cite{CG22}.
\begin{lemma}\label{lem:extrapolation}
    Let $1<p<\infty$, and suppose that $\mathcal{C}$ is a collection of function pairs $(f,g)$ with $f, g\in L^1_{loc}(\R^n)$, and for any $\ww \in A_p$, there exists a constant $C$ that depends on $[\ww]_{A_p}$ such that 
    \begin{equation}
        \int_{\R^n} \abss{f}^p\, \ww \dif x \leq C \int_{\R^n} \abss{g}^p\, \ww \dif x
    \end{equation} holds true for any $(f,g)\in \mathcal{C}$. Then for any Young function $\varphi$ with $\varphi, \varphi^{\ast} \in \Delta_2$, there exists a constant $C_e = C_e(p,\Delta_2(\varphi),\Delta_2(\varphi^{\ast}))>0$ such that we have, for any $(f,g)\in \mathcal{C}$,
    \begin{equation}
        \int_{\R^n} \varphi(\abss{f})\dif x \leq C_e \int_{\R^n}\varphi(\abss{g})\dif x.
    \end{equation}
\end{lemma}

The following result was essentially proved in \cite{Raita19,GR22,guerra2022compensated}. We sketch a proof for the convenience of the reader.
\begin{theorem}[Full space Hodge decomposition]\label{thm:full_space_hodge}
Let $1<p<\infty$.
\begin{enumerate}
\item\label{it:hodge_ABC} Let $\A,\,\B,\C$ be operators which satisfy the assumptions of Corollary~\ref{cor:ABC}. Then, each $v\in L^p(\R^n,V)$ can be decomposed as
$$
v=\B u+\A^*w
$$
where $u\in \dot{W}{^{k,p}}(\R^n,U)$, $w\in \dot{W}{^{h,p}}(\R^n,W)$ are such that $\C^*u=0$ and 
$$
\|D^ku\|_{L^p(\R^n)}+\|D^hw\|_{L^p(\R^n)}\leq C\|v\|_{L^p(\R^n)}.
$$

\item\label{it:hodge_BC} Let $\B$ and $\C$ be operators that satisfy the assumptions of Theorem~\ref{thm:BC}. We have that each $\varphi\in C_c^\infty(\R^n,U)$ can be written as
$
\varphi=\tilde \varphi + \psi,
$
where $\C^{\ast}\tilde \varphi=0$, $\B\tilde \varphi=\B\varphi$, and $\tilde\varphi\in \dot{W}{^{k,p}}(\R^n,U)$ with
$$
\|D^k\tilde\varphi\|_{L^p({\mathbb R^n})}\leq C\|\B\varphi\|_{L^p({\mathbb R^n})}.
$$
\end{enumerate}
\end{theorem}
\begin{proof}[Sketch]
    We know that for constant rank operators $\A$ we can write
    $$
    \proj_{\rmim \A(\xi)}=\A(\xi)\A^{\dagger}(\xi),
    $$
    where $\A^\dagger$ is smooth in $\R^n\setminus\{0\}$ and $(-h)$-homogeneous, see \cite{FonMul99,Raita19}. Using this fact, elementary linear algebra, and the exact relations \eqref{eq:exact_A} we can write
    \begin{align*}
        \hat v=\B(\xi)\B^\dagger(\xi)\hat v+\A^*(\xi)\A^{*\dagger}(\xi)\hat v,
    \end{align*}
    so we can define 
    $$
    \hat u=\B^\dagger(\xi)\hat v,\quad \hat w=\A^*(\xi)\A^{*\dagger}(\xi)\hat v.
    $$
    The estimates follow from
    $$
    \widehat{D^ku}=\B^{\dagger}\left(\frac{\xi}{|\xi|}\right)\hat v\otimes \left(\frac{\xi}{|\xi|}\right)^{\otimes k}
    $$
and the analogous identity for $w$ and   the H\"ormander--Mikhlin multiplier theorem. Finally, since $\rmim \B^\dagger(\xi)=\rmim \B^*(\xi)$, we also have that $\C^*(\xi)\hat u=0$, so $\C^*u=0$, proving part~\ref{it:hodge_ABC}.

Part~\ref{it:hodge_BC} follows from part~\ref{it:hodge_ABC} by replacing $\A$ with $\C^*$ and $\B$ with $\B^*$.
\end{proof}

\subsection{Auxiliary results}
 \par We first recall some properties of the $V$-function  defined as $V_p\col X \to X$ by $V_p(z):= (1+\abss{z}^2)^{\frac{p-2}{4}}z$, where $X$ is any finite dimensional inner-product space. We may use $V(z)$ for simplicity when the exponent $p$ is clear from the context.
    \begin{lemma}\label{lem:Vest} For any $z,w \in X$, we have the following:
        \begin{enumerate}[label=\textnormal{(\alph*)}]
            \item\label{it:Vest} $\abss{V_p(z)}^2 \sim_p (\abss{z}^2+\abss{z}^p)\sim_p \max\{\abss{z}^2,\abss{z}^p\}$ if $p\geq 2$, and $\abss{V_p(z)}^2 \sim_p \min\{\abss{z}^2,\abss{z}^p\}$ if $1<p<2$;
            \item\label{it:Vsum} $\abss{V_p(z+w)} \lesssim_p \abss{V_p(z)+V_p(w)}$;
            \item\label{it:Vmulti} $\abss{V_p(tz)}\leq \max\{t,t^{\frac{p}{2}}\}\abss{V_p(z)}$ for any $t>0$;
            \item\label{it:Vcvx} $\abss{V_p(\cdot)}^2$ satisfies Jensen's inequality: 
                \begin{equation*}
                    \abs{V_p \brac{\fint_{\omega} g\dif x} }^2 \leq  \fint_{\omega} \abss{V_p(g)}^2 \dif x
                \end{equation*} for any bounded domain $\omega \subset \R^n$ and any $g \in L^p(\omega)$. 
        \end{enumerate}
    \end{lemma} 
    \begin{proof}
    Properties \ref{it:Vest} and \ref{it:Vmulti} can be easily proved in the case $p\geq 2$, and see Lemma 2.1 in \cite{CFM98} for $1<p<2$. 
    \par It is easy to check that $\abss{V_p(\cdot)}^2$ is convex, which implies \ref{it:Vsum} and \ref{it:Vcvx}.
    \end{proof}

    \par For any $z,w \in V$ and any $x_0 \in \Omega$, define the shifted integrand by \[ f_w(x_0,z):= f(x_0,z+w)-f(x_0,w)-\partial_z f(x_0,w)\cdot z,\] where the function $f\col \Omega\times V \to \R$ satisfies \ref{it:fgrowth}-\ref{it:fcontinuity}.
    \begin{lemma} Suppose that the integrand $f\col \Omega\times V \to \R$ satisfies \textnormal{\ref{it:fgrowth}-\ref{it:fcontinuity}}. Given any $w\in V$ with $\abss{w}\leq M$ for some $m>0$, the following estimates hold true:
            \begin{align}
               & \abs{f_w(x_0,z)} \leq C(n,p,\dim V,M,f)\abss{V_p(z)}^2, \label{eq:pshiftedup}\\
               & \abss{f_w(x_0,z)-f_w(x_0,y)} \leq C(n,p,\dim V,M,f) (1+\abss{z}^2 +\abss{y}^2)^{\frac{p-2}{2}}\abss{z+y}\abss{z-y}, \label{eq:pshiftedLip}\\
               & \int_{B_r} f_w(x_0,z+\varphi) \dif x \geq C(p,\ell, M) \int_{B_r} \abss{V_p(\varphi)}^2 \dif x, \label{eq:pshiftedqc}
            \end{align} for any $\varphi \in C_c^{\infty}(B_r,V)$ with $\A \varphi =0$. Moreover, if we consider the $\B$-derivative of a map, the inequality \eqref{eq:pshiftedqc} holds true for any $\varphi = \B \phi$ with $\phi \in W^{k,p}_0(B_r,U)$.
    \end{lemma}
    \begin{proof}
        Inequality \eqref{eq:pshiftedup} can be proved by considering the two cases $\abss{z}\geq 1$ and $\abss{z}<1$ separately. In the first case, the Lipchitz continuity \eqref{eq:fLip} of $f(x_0,\cdot)$ implies
            \begin{align*}
                 \abss{f_w(x_0,z)} &=\abs{\int_0^1 \partial_z f(x_0,w+tz)\cdot z\dif t - \partial_z f(x_0,w)\cdot z} \\
                &\leq  \int_0^1 C(2+\abss{w+tz}^{p-1}+\abss{w}^{p-1})\dif t \abss{z} \leq C\abss{V_p(z)}^2.
            \end{align*} In the second one, by Taylor's theorem we have
            \begin{equation*}
                \abss{f_w(x_0,z)} = \abs{\int_0^1 (1-t) \partial_z^2 f(x_0,w+tz) [z,z]\dif t} \leq C\abss{z}^2 \leq C \abss{V_p(z)}^2.
            \end{equation*}
        \par To show \eqref{eq:pshiftedLip}, we need the estimate \[\abss{\partial_z f_w(x_0,z)} \leq C (1+\abss{z}^2)^{\frac{p-2}{2}}\abss{z}.\] The latter can be proved as well by considering $\abss{z}\geq 1$ and $\abss{z}<1$ separately. Then \eqref{eq:pshiftedLip} follows from this estimate and Lemma \ref{lem:comp0}.
        \par The strong $\A$-quasiconvexity \ref{it:fqcstrong} of $f$ implies \eqref{eq:pshiftedqc}. Then use the Lipschitz continuity of $f(x_0,\cdot)$ to extend it to $\varphi = \B \phi$ with $\phi \in W^{k,p}_0(B_r,U)$.
    \end{proof}
    \par We conclude this section with two auxiliary results that will be used later. The first one is an iteration lemma used in the proof of the Caccioppoli-type inequality \ref{subsec:caccioppoli} and is adapted from Lemma 6.1 in \cite{Giusti03}.
    \begin{lemma}\label{lem:iteration}
        Suppose that the function $\Phi\col (0,R] \to \R_+$ is bounded, and $\Psi_i \col (0,R] \to \R_+, i=1,\dots, k$, are decreasing with $\Psi_i(\sigma t) \leq \sigma^{-\beta i}\Psi_i(t)$, $\beta>0$, for any $t \in (0,R]$ and any $\sigma \in (0,1)$. If the following inequality 
            \begin{equation}\label{eq:iteraineq}
                \Phi(r) \leq \theta \Phi(s) + \sum_{i=1}^k \Psi_i(s-r) +B
            \end{equation} holds true for any $r,s \in (0,R]$ with $r <s$ and some given $\theta \in (0,1)$, $B>0$, then there exists $C=C(\theta, k, \beta,\tau)>0$ for any $\tau \in (0,1)$ such that $\Phi(\tau R)$ is controlled as follows:
            \begin{align*}
                \Phi(\tau R) \leq C \brac{\sum_{i=1}^k \Psi_i((1-\tau)R)+B}.
            \end{align*}
    \end{lemma}
    \begin{proof}
        Take $\lambda \in (0,1)$ to be determined later, and set \[ r_0 =\tau R,\qquad r_{j+1}-r_j = (1-\lambda)\lambda^j(1-\tau)R \quad \mbox{for }j\in \N. \] Then we apply (\ref{eq:iteraineq}) to $r_j,r_{j+1}$ iteratively with $\tau^{\prime}:= 1-\tau$ and obtain
		\begin{align*}
		\Phi(r_0) &\leq \theta \Phi(r_1)+\sum_{i=1}^k\Psi_i((1-\lambda)\tau^{\prime}R) +B\\
		&\leq \theta (\theta \Phi(r_2)+\sum_{i=1}^k\Psi_i((1-\lambda)\lambda\tau^{\prime}R)+B)+\sum_{i=1}^k\Psi_i((1-\lambda)\tau^{\prime}R)+B\\
		&\leq \dots \leq \theta^{L+1}\Phi(r_{L+1}) +\sum_{j=0}^L \theta^j \brac{\sum_{i=1}^k\Psi_i((1-\lambda)\lambda^j\tau^{\prime}R)+B}.
		\end{align*} Since $\Psi_i((1-\lambda)\lambda^j\tau^{\prime}R) \leq (\tau^{\prime})^{-\beta i}(1-\lambda)^{-\beta i}\lambda^{-\beta ij}\Psi_i(R)$, under the condition $\theta < \lambda^{\beta i}, i=1,\dots, k$, the above inequality becomes 
            \begin{align*}
            \Phi(r_0) &\leq \theta^{L+1}\Phi(r_{L+1}) +  \sum_{i=1}^k \Psi_i(R)(\tau^{\prime})^{-\beta i}(1-\lambda)^{-\beta i} \brac{\sum_{j=0}^L \frac{\theta^j}{\lambda^{\beta ij}}} + \frac{1-\theta^{L+1}}{1-\theta}B \\ 
            &\leq \theta^{L+1}\Phi(r_{L+1}) + \sum_{i=1}^k \Psi_i(R) (\tau^{\prime})^{-\beta i}(1-\lambda)^{-\beta i}  \frac{1-(\theta \lambda^{-\beta i})^{L+1}}{1-\theta\lambda^{-\beta i}}+\frac{1-\theta^{L+1}}{1-\theta}B.
            \end{align*}
            The desired result can be obtained by taking $L\to \infty$. Notice that if we do not take $\tau'$ out of $\Psi_i$, the inequality obtained will be 
            \begin{equation}
                \Phi(\tau R) \leq C \brac{\sum_{i=1}^k \Psi_i((1-\tau)R)+B}.
            \end{equation}
    \end{proof} 
\par The second one is Lemma 8.3 in \cite{Giusti03}, and we include it here for completeness:
    \begin{lemma}\label{lem:comp0}
    Suppose that $X$ is a finite-dimensional vector space and $z,w\in X$ with $z\neq w$. Then for any $-1<q<\infty$ and $s,\gamma \geq 0$, there exists $c_1 =c_1(q,s,\gamma)\geq 1$ such that 
		\begin{multline}\label{eq:comp0}
		 \frac{1}{c_1}(\gamma^2+\abs{z}^2+\abs{w}^2)^{\frac{q}{2}} \leq \int_0^1 (1-t)^s(\gamma^2+\abs{tz+(1-t)w}^2)^{\frac{q}{2}}\dif t \\
   \leq c_1 (\gamma^2+\abs{z}^2+\abs{w}^2)^{\frac{q}{2}}.
		 \end{multline}
    \end{lemma} 

\section{Mean coercivity and lower semicontinuity}\label{sec:coercivity}
We next prove that strong quasiconvexity implies mean coercivity for autonomous integrands, under standard growth conditions.
\begin{theorem}\label{thm:Acoercivity} 
    Suppose that function $f \col V \to \R$ satisfies \textnormal{\ref{it:fgrowth}} and \textnormal{\ref{it:fqcstrong}} with $1<p<\infty$. Then there exist constants $a_i \in \R$, $i=1,2,3$, with $a_1=a_1(p,\ell)>0$ and $a_2 =a_2(n,\dim V, p,f(0),L,\ell)$ and $a_3 = a_3(n,\dim V, p, L,\ell)>0$ such that, for any bounded open set $\Omega \subset \R^n$ and $\A$-free map $v_0 \in L^p(\Omega,V)$, the following holds:
        \begin{equation}\label{eq:Acoercivity}
            \fint_{\Omega}f(v_0 +\varphi)\dif x \geq a_1\fint_{\Omega}\abss{\varphi}^p \dif x +a_2 -a_3 \fint_{\Omega} \abss{v_0}^p \dif x
        \end{equation} for any $\varphi \in C_{c,\A}^\infty(\Omega)$. 
\end{theorem}
\begin{remo}
    We do not know whether the result of Theorem~\ref{thm:Acoercivity} holds for non-autonomous functionals. In fact, this problem is open even for functionals defined on gradient fields: Suppose that $f\colon \Omega\times \R^{N\times n}$ is quasiconvex of $p$-growth. Then the question is whether there exist constants $a_i$, $i=1,2,3$, with $a_1>0$ such that for any $u_0\in W^{1,p}(\Omega,\R^N)$
     \begin{equation*}
            \fint_{\Omega}f(x,D u_0 +D\phi)\dif x \geq a_1\fint_{\Omega}\abss{D\phi}^p \dif x +a_2 -a_3 \fint_{\Omega} \abss{u_0}^p+\abss{Du_0}^p \dif x
        \end{equation*} for any $\phi \in C_{c}^\infty(\Omega,\R^N)$. 
\end{remo}
\begin{proof}
     First, by \ref{it:fqcstrong}, the strong $\A$-quasiconvexity of $f$, we have
        \begin{align}
             \int_{\Omega}f(v_0+\varphi)\dif x &= \int_{\Omega} f(\varphi)\dif x +\int_{\Omega}(f(v_0+\varphi)-f(\varphi))\dif x \label{eq:Acoermid1} \\
            &\geq \int_{\Omega}(\ell \abss{V_p(\varphi)}^2 +f(0))\dif x +\int_{\Omega}\int_0^1 \partial_z f (\varphi +tv_0)\dif t \cdot v_0\dif x \notag \\
            &=: I + II. \notag
        \end{align}
    \par The two terms $I$ and $II$ are estimated separately in the following. Since we know that  $\abss{z}^p \lesssim_p \abss{V_p(z)}^2+1$ by \ref{it:Vest}, the first term can be bounded from below by 
        \begin{equation}\label{eq:Acoermid2}
             I \geq  C_1\int_{\Omega} \abss{\varphi}^p \dif x+(f(0)-C_1^{\prime})\Le^n(\Omega),
        \end{equation} where $C_1, C_2>0$ depending on $p,\ell$.
    \par To control the second term, the Lipschitz continuity of $f$ in $V$ (see \eqref{eq:fLip}) is needed. With this bound of $\partial_z f$ and Young's inequality, we have
        \begin{align}
            \abss{II} &\leq CL\int_{\Omega}\int_0^1 (1+\abss{\varphi + t v_0}^{p-1})\abss{v_0}\dif t \dif x \notag\\
            &\leq C\int_{\Omega} (1+\abss{\varphi}^{p-1}\abss{v_0} +\abss{v_0}^p)\dif x \label{eq:Acoermid3}\\
            &\leq C_2 \int_{\Omega} (1+\varepsilon\abss{\varphi}^p +C(\varepsilon,p)\abss{v_0}^p)\dif x, \notag
        \end{align} where $C_2 = C_2(n,\dim V,p,L)>0$. Take $\varepsilon$ to be small enough such that $\varepsilon C_2 <\frac{C_1}{2}$, then it follows from \eqref{eq:Acoermid1}, \eqref{eq:Acoermid2} and \eqref{eq:Acoermid3} that 
            \begin{align*}
                \fint_{\Omega}f(v_0 +\varphi)\dif x &\geq C_1\fint_{\Omega} \abss{\varphi}^p \dif x +(f(0)-C_1^{\prime}) \\
                & -C_2\brac{\varepsilon\fint_{\Omega}\abss{\varphi}^p\dif x +C(\varepsilon,p)\fint_{\Omega} \abss{v_0}^p\dif x +1} \\
                & \geq \frac{C_1}{2}\fint_{\Omega} \abss{\varphi}^p\dif x +C_3 - C_2C(\varepsilon,p)\fint_{\Omega}\abss{v_0}^p \dif x.
            \end{align*} The desired inequality \eqref{eq:Acoercivity} then follows with $a_1 =\frac{C_1}{2}$, $a_2 =C_3$ and $a_3 = C_2C(\varepsilon,p)$.
\end{proof}

\begin{remo}\label{rmk:Bcoercivity}
    Notice that for any $g \in \mathscr D'(\Omega,U)$ with $\B g\in L^p(\Omega,V)$ and $\phi \in C_c^\infty(\Omega,U)$, the following mean coercivity also holds true by the same argument:
        \begin{equation*}
            \fint_{\Omega}f(\B (g+\phi))\dif x \geq a_1 \fint_{\Omega} \abss{\B \phi}^p\dif x + a_2 -a_3 \fint_{\Omega} \abss{\B g}^p \dif x.
        \end{equation*}
\end{remo}

\begin{lemma}\label{lem:lsc}
    Suppose that $f\colon \Omega \times V\to \R$ is $\A$-quasiconvex (as in \textnormal{\ref{it:fqcstrong}} with $\ell=0$) and satisfies \textnormal{\ref{it:fgrowth}} for $1<p<\infty$. Then  for any $v_0\in L^p(\Omega,V)$ and any $C_{c,\A}^\infty(\Omega)\ni v_j\weakto v$ in $L^p(\Omega,V)$ we have that
    $$
    \liminf_{j\to\infty}\int_\Omega f(x,v_0+v_j)\dif x\geq \int_\Omega f(x,v_0+v) \dif x.
    $$
\end{lemma} 
\par The reason why we isolate this result as a separate lemma is to highlight that it does not follow from the main results of \cite{FonMul99}, which does not cover signed integrands explicitly. The lemma also does not follow by an immediate application of the lower semicontinuity results in \cite{GR22} either, as those do not cover effects of concentration at the boundary. The simple argument below shows that this is not an issue in problems with Dirichlet boundary conditions.
\begin{proof}
    Consider an open bounded set $\Omega'\Supset\Omega$. Extend $v_0,\{v_j\}$ and $v$ by $0$ to $\Omega^{\prime}\setminus \Omega$ and denote the extension by $\bar{v}_0, \{\bar{v}_j\}$ and $\bar{v}$, respectively. Also let $\bar f(x,z)=\chi_{\Omega}(x)f(x,z)$. It is easy to see that $\bar{v}_j \weakto \bar{v}$ in $L^p(\Omega^{\prime},V).$ 
    Note that $\A (\bar v_0+\bar{v}_j)=\A \bar v_0\in W^{-h,p}_{\locc}(\Omega',W)$, so Theorem~5.6 in \cite{KriRai20} applies in $\Omega'$, giving
  $$
    \liminf_{j\to\infty}\int_{\Omega'} \bar f(x,\bar v_0+\bar{v}_j) \dif x\geq \int_{\Omega'} \bar f(x,\bar v_0+\bar{v})\dif x.
    $$
    The claim is then immediate.
\end{proof}

\section{Existence of generalized minimizers}\label{sec:existence}
We begin by defining our notion of generalized minimizers, which requires $L^{p}_{0,\A}(\Omega)$ and $X^{\B,p}_0(\Omega)$, the two weakly closed $L^p$-based classes defined in Section~\ref{subsec:space}.

\begin{definition}\label{def:generalized_min}
    Let $v_0\in L^p(\Omega,V)$ be such that $\A v_0=0$ (resp. $u_0\in \mathscr D'(\Omega,U)$ be such that $v_0=\B u_0\in L^p(\Omega,V)$). Let $f\colon \Omega\times V\to \R$ be measurable. We say that $v\in v_0+L^p_{0,\A}(\Omega)$ (resp. $v\in \B u_0+ X^{\B,p}_0(\Omega)$) is a \textup{\textbf{generalized minimizer}} of Problem~\eqref{eq:P_Omega_Afree} (resp.~\eqref{eq:P_Omega}) if
    \begin{align*}
        \int_\Omega f(x,v(x))\dif x\leq \int_\Omega f(x,v_0(x)+\varphi(x))\dif x\quad\text{for all }\varphi\in L^p_{0,\A}(\Omega)\text{ (resp. $X^{\B,p}_0(\Omega)$)}.
    \end{align*}
\end{definition}
\subsection{$\A$-free setting}\label{subsec:A_exist}
We begin by proving that $\A$-quasiconvexity and mean coercivity imply existence of minimizers of Problem~\eqref{eq:P_Omega_Afree} by the direct method.
\begin{proposition}\label{prop:coerc_implies_existence_A-free}
     Suppose that $f\colon \Omega \times V\to \R$ is $\A$-quasiconvex (\textnormal{\ref{it:fqcstrong}} with $\ell=0$) and satisfies \textnormal{\ref{it:fgrowth}} for $1<p<\infty$. Assume in addition that there exists $C>0$ such that for any $v_0\in L^{p}(\Omega,V)$, $\varphi\in C_c^\infty(\Omega,V)$ with $\A v_0=0=\A\varphi$, we have that 
     \begin{align}\label{eq:1}
     \fint_{\Omega}f(x,v_0+\varphi)\dif x\geq C\left(\fint_\Omega |\varphi|^p-1-\fint_\Omega|v_0|^p\right).
     \end{align}
     Then Problem~\eqref{eq:P_Omega_Afree} admits a generalized minimizer in the sense of Definition~\ref{def:generalized_min}.
\end{proposition}

\begin{proof} 
Let $v_0\in L^{p}(\Omega,V)$ be $\A$-free. We claim that 
$$
\inf\left\{\int_{\Omega}f(x,v_0+v)\colon v\in C_{c,\A}^\infty(\Omega)\right\}=\min \left\{\int_{\Omega}f(x,v_0+v)\colon v\in L_{0,\A}^{p}(\Omega)\right\}.
$$
First, it is obvious that the inequality ``$\geq$'' holds. Second, we show that the minimization problem on the right has a solution in $L_{0,\A}^{p}(\Omega)$. To see this note that by \eqref{eq:1}, we can find a minimizing sequence $\{v_j\}\subset L_{0,\A}^{p}(\Omega)$ which is bounded in $L^p(\Omega)$ by the assumed coercivity condition \eqref{eq:1}. It follows by the $\A$-quasiconvexity of $f$ and Lemma \ref{lem:lsc} that, up to a subsequence of $\{v_j\}$, 
$$
\liminf_{j\to\infty} \int_{\Omega}f(x,v_0+v_j) \dif x\geq \int_{\Omega}f(x,v_0+\tilde v)\dif x,
$$
where $\tilde v$ is a weak limit point of $\{v_j\}$ and thus  $\tilde v\in L_{0,\A}^{p}(\Omega)$ by Mazur's Lemma. The map $\tilde v$ is thus a minimizer of the problem on the right. 

Let now $\phi_j\in C_{c,\A}^\infty(\Omega)$ be such that $\phi_j\to \tilde v $ in $L^p(\Omega,V)$. The Lipschitz continuity of $f$ in \eqref{eq:fLip}  implies
$$
\int_\Omega f(x,v_0+ \phi_j)\dif x\to\int_\Omega f(x,v_0+ \tilde v)\dif x,
$$
which shows that $\{\phi_j\}$ is a minimizing sequence for the problem on the left, so the claimed equality holds true.
\end{proof}
\begin{proof}[Proof of Theorem~\ref{thm:existence_autonomous}, $\A$-free setting]\phantom{a}\\
    Follows immediately from Proposition~\ref{prop:coerc_implies_existence_A-free} and Theorem~\ref{thm:Acoercivity}.
\end{proof}
We proceed to show that any minimizer $v_0+\tilde v$ of \eqref{eq:P_Omega_Afree} has a far better structure. In particular, the following will be deduced in Proposition~\ref{prop:reduction_A-free}: as far as problem $\eqref{eq:P_Omega_Afree}$ is concerned, we can assume without loss of generality that $\Omega$ is smooth and that the minimizers have the form $v=\phi+\B u$, where $\phi\in C^{\infty}(\bar\Omega,V)$ is $\A$-free and $u \in W^{k,p}(\Omega,U)$.

\begin{proposition}\label{prop:reduction_A-free}
    Let $v_0+\tilde v$ be a minimizer of \eqref{eq:P_Omega_Afree}, as described above, and let $\omega\Subset\Omega$. We have that
       \begin{align*}
            \inf \{\mathcal E(v,\omega)\colon v\in  v_0+\tilde v+C_{c,\A}^\infty(\omega)\}=\mathcal E(v_0+\tilde v,\omega),
        \end{align*} i.e. that $v_0+\tilde{v}$ is a local minimizer of \eqref{eq:P_Omega_Afree}. Moreover, it is possible to write $v_0+\tilde v= \B\tilde u+\phi$ in $\omega$, where
        \begin{enumerate}
            \item $\tilde u\in {\dot{W}}{^{k,p}}(\R^n, U)$ satisfies $\C^{\ast}\tilde u=0$ in $\R^n$,
            \item $\phi\in L^p(\R^n,V)$ satisfies $\B^{\ast} \phi=0$ in $\R^n$ and $\A \phi=0$ in $\bar\omega$.
        \end{enumerate}
        In particular, $\phi\in C^\infty(\bar\omega,V)$.
\end{proposition}
\begin{proof}
    \par The first conclusion is easy to check by testing with functions in $C_{c,\A}^{\infty}(\omega)$.
   \par To show the rest, take $\rho\in C_c^\infty(\Omega)$ be equal to $1$ in an open neighborhood of $\bar \omega$ and perform the Hodge decomposition of $\rho(v_0+\tilde{v})$ as in Theorem \ref{thm:full_space_hodge}\ref{it:hodge_ABC}: 
   \[\rho(v_0+\tilde{v})=\B\tilde u+\phi,\] 
   where $\tilde{u} \in \dot{W}^{k, p}(\R^n,U)$  with $\C^*\tilde u=0$ in $\R^n$ and $\phi = \A^{\ast}\psi$ for some $\psi \in \dot{W}^{h, p}(\R^n,W)$. Therefore $\B^*\phi=\B^*\A^*\psi=0$ in $\R^n$.
   Since 
        \[ \A (v_0+\tilde{v}) = 0  \mbox{ in }\Omega \quad \mbox{and}\quad \A \B \tilde{u}=0\text{ in }\R^n,\] 
    we obtain $\A \phi =0$ in $\bar{\omega}$.

   The assertion concerning the smoothness (in fact, analyticity) of $\phi$ follows from the fact that the system 
   \[\A \phi=0,\quad \B^{\ast}\phi=0\] is elliptic. 
   We verify the ellipticity explicitly here: let $\xi\neq 0$, so that  $\ker \A(\xi)\cap\ker\B^{\ast}(\xi)=\rmim \B(\xi)\cap(\rmim\B(\xi))^\bot=\{0\}$.
\end{proof}
\begin{remo}\label{rmk:reduction}
In particular, we obtain that $\tilde u\in W^{k,p}(\omega,U)$ is $\C^{\ast}$-free and that
\begin{align*}
     \inf \left\{ \int_{\omega} f(x,\B u+\phi)\colon u\in \tilde u+C_c^\infty(\omega,U)\right\}= \int_{\omega}f(x,\B \tilde u+\phi),
\end{align*}
where $\phi\in C^\infty(\bar\omega,V)$ is $\A$-free and $\B^{\ast}$-free in $\bar\omega$. This shows that the regularity claim of Theorem~\ref{thm:main_A-free} can be inferred from the regularity claim of Theorem~\ref{thm:main} by considering the integrand $\tilde f(x,z)=f(x,z+\phi(x))$ for $x\in\omega$ and $z\in V$.
\end{remo}

\subsection{$\B$-gradient setting}\label{susbsec:B_exist}
In this subsection, we first show the existence of a generalized minimizer for \eqref{eq:P_Omega} under quasiconvexity and coercivity assumptions.
\begin{proposition}\label{prop:coerc_implies_existence_B}
     Suppose that $f\colon \Omega \times V\to \R$ is quasiconvex for $\B$-gradients (\textnormal{\ref{it:f_B-qcstrong}}$'$ with $\ell=0$) and satisfies \textnormal{\ref{it:fgrowth}} for $1<p<\infty$. Assume in addition that for any $u_0\in \mathscr D'(\Omega,U)$  such that $\B u_0\in L^p(\Omega,V)$ and $ \phi\in C_c^\infty(\Omega,U)$, we have that 
     \begin{align*}
     \fint_{\Omega}f(x,\B(u_0+\phi))\geq C\left(\fint_\Omega |\B\phi|^p-1-\fint_\Omega|\B u_0|^p\right).
     \end{align*}
     Then Problem~\eqref{eq:P_Omega} admits a generalized minimizer in the sense of Definition~\ref{def:generalized_min}.
\end{proposition}
The proof is analogous to that of Proposition~\ref{prop:coerc_implies_existence_A-free}, establishing that 
$$
\inf\left\{\int_{\Omega}f(x,\B (u_0+u))\colon u\in C_c^\infty(\Omega,U)\right\}=\min \left\{\int_{\Omega}f(x,\B u_0+v)\colon v\in X_0^{\B,p}(\Omega)\right\}.
$$
We do not repeat the argument here. 
\begin{proof}[Proof of Theorem~\ref{thm:existence_autonomous}, $\B$-gradient setting]\phantom{a}\\
        Follows immediately from Proposition~\ref{prop:coerc_implies_existence_B} and Remark~\ref{rmk:Bcoercivity}.
\end{proof}
We will show in addition that any minimizer $\B u_0+\tilde v$ has a better structure. In Proposition~\ref{prop:reduction} we will  deduce that, as far as problem \eqref{eq:P_Omega} is concerned, we can assume without loss of generality that $\Omega$ is smooth and $u_0\in W^{k,p}(\Omega,U)$, reducing our setup of Theorem \ref{thm:main} to the one in \cite{Dacorogna_JEMS}.

We begin by choosing a good representative $\tilde u$ for $\tilde v=\B\tilde u$:
\begin{proposition}\label{prop:existence}
Let $1<p<\infty$ and $f$ satisfy the assumptions of Proposition~\ref{prop:coerc_implies_existence_B}. For any $u_0\in \mathscr D'(\Omega,U)$ with $\B u_0\in L^p(\Omega,V)$, there exists $\tilde u_0\in \mathscr D'(\R^n,U)$ such that
    \begin{enumerate}
        \item $\B \tilde u_0$ is a generalized minimizer of \eqref{eq:P_Omega},
        \item $\tilde u_0-u_0\in W^{k,p}(\Omega,U)$,
        \item $\C^{\ast}(\tilde u_0-u_0)=0$.
    \end{enumerate}
\end{proposition}
\begin{proof}
 Let $u_0+\varphi_j$ with $\varphi_j\in C_c^\infty(\Omega,U)$ be a minimizing sequence. By the mean coercivity assumption we have the $L^p$-boundedness of $\{\B \varphi_j\}$. Next, project $\varphi_j$ to the kernel of $\C^{\ast}$  as in Theorem~\ref{thm:full_space_hodge}\ref{it:hodge_BC} to get 
$$
\varphi_j=\tilde \varphi_j+\psi_j,
$$
where $\tilde{\varphi}_j \in \dot{W}^{k,p}(\R^n)$ and $\C^{\ast}\tilde \varphi_j=0$, $\B\tilde \varphi_j=\B\varphi_j$ in ${\mathbb R^n}$. Moreover, we have 
    \[\normm{D^k \tilde{\varphi}_j}_{L^p(\R^n)} \leq C \normm{\B \varphi_j}_{L^p(\R^n)}.\] 
Since $\Omega \subset \R^n$ is bounded, the Sobolev estimate for $\tilde\varphi_j$ implies the boundedness of $\{\tilde\varphi_j\}$ in $W^{k,p}(\Omega,U)$. Then there exists $\pi \in W^{k,p}(\Omega,U)$ with $\C^{\ast}\pi=0$ such that, up to a subsequence,
$$
\tilde \varphi_j\rightharpoonup \pi \text{ in }W^{k,p}(\Omega,U).
$$
Write $\tilde u_0=u_0+\pi$, and it is easy to see that $\B (u_0+\varphi_j) = \B(u_0+\tilde{\varphi}_j)\rightharpoonup \B \tilde u_0$ in $L^p(\Omega,V)$. By Lemma \ref{lem:lsc}
$$
\liminf_{j\to\infty}\int_\Omega f(x,\B(u_0+\tilde\varphi_j))\geq \int_\Omega f(x,\B\tilde u_0).
$$
Notice that the left-hand side converges to the infimum of \eqref{eq:P_Omega} since $\B\tilde \varphi_j=\B\varphi_j$, and $\B\tilde{u}_0 \in \B u_0 +X_0^{\B,p}(\Omega)$, which indicates \[ \int_\Omega f(x,\B\tilde u_0)\dif x \geq \liminf_{j\to\infty}\int_\Omega f(x,\B(u_0+\tilde\varphi_j)) \dif x, \] and thus the equality holds. The claim follows.
\end{proof}

\begin{lemma}\label{lem:smooth_domains}
    We follow the notation in Proposition~\ref{prop:existence}. Let $\omega\Subset\Omega$, and then there holds 
        \begin{align*}
            \inf \{\F(u,\omega)\colon u\in \tilde u_0+C_c^\infty(\omega,U)\}=\F(\tilde u_0,\omega).
        \end{align*} In other words, $\tilde{u}_0$ is a local minimizer of \eqref{eq:P_Omega}.
\end{lemma}
In particular, to prove Theorem~\ref{thm:main} it suffices to consider functionals defined on smooth sets.

    \begin{proof}[Proof of Lemma~\ref{lem:smooth_domains}]
        Fix an arbitrary $\varphi \in C_c^{\infty}(\omega,{U})$. Since $u_0+\pi$ is a generalized minimizer, and $\B(u_0+\pi+\varphi) \in \B u_0 + X_0^{\B,p}(\Omega)$, we have \[ \F(u_0+\pi,\Omega) \leq \F(u_0+\pi+\varphi,\Omega), \] which implies \[ \int_{\omega}f(x,\B(u_0+\pi))\dif x \leq \int_{\omega}f(x,\B(u_0+\pi+\varphi))\dif x \] as desired.
    \end{proof}
    Finally, we show that on any subdomain $\omega\Subset\Omega$, we can replace the minimizer with a minimizer in the right Sobolev class, provided that we assume integrability of the lower order derivatives of $u_0$.
    \begin{proposition}\label{prop:reduction}
        We use the notation of Proposition \ref{prop:existence} and denote $u_0+\pi$ by $\tilde{u}_0$. Assume that $u_0\in W^{k-1,p}(\Omega,U)$. For any $\omega\Subset \Omega$, there exists $u_\omega\in W^{k,p}(\Omega,U)$ such that
        \begin{enumerate}
            \item $\B u_\omega=\B \tilde u_0$ in $\omega$,
            \item $\C^{\ast}u_{\omega}=0$ in $\Omega$,
            \item $\|u_{\omega}\|_{W^{k,p}(\Omega)}\leq c\|\B u_0\|_{L^{p}(\Omega)}+\|u_0\|_{W^{k-1,p}(\Omega)}$,
            \item $u_\omega$ is a minimizer of the problem
            \begin{align*}
                \inf \{\F(u,\omega)\colon u\in  u_\omega+C_c^\infty(\omega,U)\}=\F( u_\omega,\omega).
            \end{align*}
        \end{enumerate}
    \end{proposition}
    \begin{proof}
        Consider a cut-off $\rho\in C_c^\infty(\Omega)$ such that $\rho=1$ in $\omega$. We now apply Theorem~\ref{thm:full_space_hodge}\ref{it:hodge_BC} to decompose $\rho \tilde{u}_0$ as follows: 
        \[\rho \tilde u_0=u_\omega+\psi,\] 
        where $u_{\omega} \in \dot{W}^{k,p}(\R^n,U)$ such that $\C^{\ast}u_{\omega}=0$ and $\B u_{\omega} = \B (\rho \tilde u_0) \in L^p(\R^n,V)$.  {The third assertion follows from the estimate of Theorem}~\ref{thm:full_space_hodge}\ref{it:hodge_BC} and the product rule. The last claim follows from $\B u_\omega=\B \tilde u_0$ in $\omega$ and Lemma~\ref{lem:smooth_domains}. 
    \end{proof}
    In particular, the conclusion of this subsection is that we can assume that $\Omega$ is smooth, $u_0\in W^{k,p}(\Omega, U)$, and that minimizers are of class $W^{k,p}(\Omega,U)$ and $\C^*$-free in the statement of Theorem \ref{thm:main}.

\section{Local linear estimates}\label{sec:linear_inequ}
\subsection{Local Korn-type inequality} 
We will establish some linear estimates for the minimizers of Problem \eqref{eq:P_Omega} which are  a generalization of Lemma~2.10 in \cite{RaitaTAMS}. Below, $\B$ and $\C$ are as in Theorem~\ref{thm:BC} so $k,\,l$ are the orders of $\B,\,\C$ respectively.
\begin{proposition}\label{prop:Korn}
        Suppose that $\Omega \subset \R^n$ is a bounded open set, and $v \in \mathscr{D}^{\prime}(\Omega,U)$ satisfies \[\B v \in L^p_{\locc}(\Omega,V) \quad \mbox{and}\quad \C^{\ast}v \in \dot{W}_{\locc}^{-q,p}(\Omega,U)\] for some $1<p<\infty$, where $q=l-k$. Then we have $v \in W^{k,p}_{\locc}(\Omega,U)$, and for any open subsets $\omega \Subset \Omega^{\prime} \Subset \Omega$, there holds 
            \begin{equation}\label{eq:Korn}
                \normm{D^k v}_{L^p(\omega)} \leq C(\normm{\B v}_{L^p(\Omega^{\prime})} + \normm{\C^{\ast}v}_{\dot{W}^{-q,p}(\Omega^{\prime})}+ \normm{v}_{W^{k-1,p}(\Omega^{\prime})})
            \end{equation} with $C=C(n,p,\B,\C,\omega, \Omega^{\prime})>0$. In particular, if $\C^{\ast} v=0$ in $\Omega$, for any ball $B_R=B(x_0,R)\Subset \Omega$ and any $\tau \in (0,1)$, we have
            \begin{equation}\label{eq:Korn_ball}
                \begin{split}
                    \int_{B_{\tau R}} \abss{D^k v}^p\ww \dif x &\leq C_1 \brac{\int_{B_R} \abss{\B v}^p \ww \dif x + \sum_{i=0}^{k-1} \int_{B_R}\frac{\abss{D^i v}^p}{R^{p(k-i)}} \ww\dif x}\\
                    \int_{B_{\tau R}} |V_p(D^k v)|^2\dif x &\leq C_2 \brac{\int_{B_R} |V_p({\B v})|^2 \dif x + \sum_{i=0}^{k-1} \int_{B_R}\left|V_p\left(\frac{{D^i v}}{R^{k-i}}\right)\right|^2 \dif x},
                    \end{split}
            \end{equation} where $\ww$ is any $A_p$ weight, $C_1 = C(n,p,\B,\C, \tau,[\ww]_{A_p})>0$, $C_2=C_2(n,p,\B,\C, \tau)>0$, and the integrals in the first inequality may take the value $\infty$.
    \end{proposition}
    \begin{proof}
    \textbf{Part 1.} Regularity of $v$. 
    
        Take any subsets $\omega$ and $\Omega^{\prime}$ as in the assumption, an open set $\Omega^{\pprime}$ with $\omega\ssubset \Omega^{\pprime}\ssubset \Omega^{\prime}$, and a cut-off function $\rho \in C_c^{\infty}(\Omega^{\pprime})$ with $\indi_{\omega}\leq \rho \leq \indi_{\Omega^{\pprime}}$. Define $w := \rho v$ and an operator $\PP:= (D^q\B, \C^{\ast})$.
        
        \par Given any multi-index $\alpha$ with $\abss{\alpha}=k+q-1=l-1$, we have 
            \begin{align*}
             \widehat{\partial^{\alpha}w}(\xi) &= (-i)^{k+q-1}\xi^{\alpha}\widehat{w}(\xi)\\
            &=  (-i)^{k+q-1}\xi^{\alpha} \PP^{\dag}(\xi)\PP (\xi)\widehat{w}(\xi)\\
            &= i\, \xi^{\alpha} \abss{\xi}^{-(k+q)} \PP^{\dag}\brac{\frac{\xi}{\abss{\xi}}} \widehat{\PP w}(\xi)\\
            &\eqqcolon\ m(\xi) \widehat{\PP w}(\xi),
            \end{align*}where $\PP^{\dag}(\xi)$ is a left-inverse of $\PP (\xi)$, and the multiplier $m(\xi)$ is smooth away from the origin and $(-1)$-homogeneous. Let $K:= \FT^{-1}m$, which is $(1-n)$-homogeneous and integrable near the origin. Then the partial derivative $\partial^{\alpha}(\rho w) $ can be expressed as follows:
        \begin{align*}
            \partial^{\alpha}w = \partial^{\alpha}(\rho v) = K \ast (\PP (\rho v))  
            =K\ast (\rho \PP v) + \sum_{j=1}^{k+q}K\ast (b_j[D^j\rho, D^{k+q-j}v]), \notag
        \end{align*} 
        where $\{b_j[\cdot, \cdot]\}_{j=1}^{k+q}$ are bilinear forms with constant coefficients.   
        Notice that $D^j\rho \equiv 0$ on $\omega$, and thus \[\mbox{sing}\,\spt (b_j[D^j\rho, D^{k+q-j}v]) \subset \R^n \setminus \omega,\] and we have that the terms in the sum are smooth on $\omega$. To control the first term, notice that for any $\varphi \in C_c^{\infty}(\omega,U)$ there holds
        \begin{align*}
            \brangle{K\ast (\rho \PP v), \varphi} &= \brangle{\PP v,\rho (K^{\prime}\ast \varphi)}\\ 
            &= (-1)^q \brangle{\B v, (\di)^q(\rho (K^{\prime}\ast \varphi))} +\brangle{\C^{\ast}v,\rho (K^{\prime}\ast \varphi)},
        \end{align*} 
        where $K^{\prime}(\cdot) = K(-\,\cdot\,)$ is the dual convolution kernel of $K$. The second term can be easily controlled as follows 
        \begin{align*}
            \brangle{\C^{\ast}v,\rho (K^{\prime}\ast \varphi)} &\leq \normm{\C^{\ast}v}_{\dot{W}^{-q,p}(\Omega^{\pprime})} \normm{\rho (K^{\prime}\ast \varphi)}_{\dot{W}^{q,p'}(\Omega^{\pprime})} \\ 
            &\leq C\normm{\C^{\ast}v}_{\dot{W}^{-q,p}(\Omega^{\pprime})} \normm{ K^{\prime}\ast \varphi}_{W^{q,p'}(\Omega^{\pprime})}.
        \end{align*} 
        By Lemma 7.12 in \cite{GT01}, we have 
        \[ \normm{ K^{\prime}\ast \varphi}_{L^{p'}(\Omega^{\pprime})} \leq C(n)\Le^n(\Omega^{\pprime})^{\frac{1}{n}} \normm{\varphi}_{L^{p'}(\Omega^{\pprime})}, \] 
        and the estimate of the higher-order terms requires Theorem \ref{thm:HMmultiplier}: 
        \[ \sum_{i=1}^q\normm{ K^{\prime}\ast \varphi}_{\dot{W}^{i,p'}(\Omega^{\pprime})} \leq \sum_{i=1}^q\normm{ K^{\prime}\ast \varphi}_{\dot{W}^{i,p'}(\R^n)} \leq C \sum_{i=1}^q\normm{\varphi}_{\dot{W}^{i-1,p'}(\Omega^{\pprime})}.\] 
        This implies 
        \[ \normm{ K^{\prime}\ast \varphi}_{W^{q,p'}(\Omega^{\pprime})} \leq C \normm{\varphi}_{W^{q-1,p'}(\Omega^{\pprime})},\] and to obtain $\partial^{\alpha}v \in W^{-(q-1),p}_{\locc}(\Omega,U)$, we need to estimate the $L^{p'}$-norm of $(\di)^q(\rho (K^{\prime}\ast \varphi))$. This term expands as 
        \[(\di)^q(\rho (K^{\prime}\ast \varphi)) = \sum_{i=0}^q c_i[D^i \rho, D^{q-i}(K^{\prime}\ast \varphi)], \] where $\{c_i[\cdot\, ,\cdot]\}_{i=0}^q$ are bilinear forms with constant coefficients. The terms with $i<q$ can be estimated as follows
        \begin{align*}
             \left\|{\sum_{i=0}^{q-1} c_i[D^i \rho, D^{q-i}(K^{\prime}\ast \varphi)]}\right\|_{L^{p'}(\Omega^{\pprime})} 
            &\leq  C \sum_{i=0}^{q-1} \sup \abss{D^i \rho} \normm{K^{\prime}\ast\varphi}_{\dot{W}^{q-i,p{\prime}}(\R^n)}\\
            &\leq  C \sum_{i=0}^{q-1} \norm{\varphi}_{\dot{W}^{q-i-1, p'}(\R^n)}  \leq C\normm{\varphi}_{W^{q-1,p'}(\omega)},
        \end{align*}  where the second inequality follows from Theorem \ref{thm:HMmultiplier}. When $i=q$, we apply Lemma 7.12 in \cite{GT01} to estimate the corresponding term
        \begin{align*}
         \brangle{\B v, (D^q\rho) K^{\prime}\ast \varphi} 
        &\leq \ \norm{\B v}_{L^p(\Omega^{\pprime})} \normm{(D^q\rho) K^{\prime}\ast \varphi}_{L^{p'}(\Omega^{\pprime})} \\
        &\leq \ \norm{\B v}_{L^p(\Omega^{\pprime})} \sup \abss{D^q\rho} \normm{K^{\prime}\ast \varphi}_{L^{p'}(\Omega^{\pprime})}\\
        &\leq \  C(n) \Le^n(\Omega^{\pprime})^{\frac{1}{n}}\norm{\B v}_{L^p(\Omega^{\pprime})} \normm{\varphi}_{L^{p'}(\omega)}.
        \end{align*} 

    \par  From the above, we know that $\partial^{\alpha}v \in W^{-(q-1),p}_{\locc}(\Omega,U)$ for any multi-index $\alpha$ with $\abss{\alpha}=k+q-1$. Then with a procedure similar to the one above and $\PP$ replaced by $D^{k+q-i}$, we can show inductively that for any $\alpha$ with $\abss{\alpha}=k+q-i-1$, $\partial^{\alpha}v \in W^{-(q-i-1),p}_{\locc}(\Omega,U)$, for $i=1,\dots,q-1$. Furthermore, with $\PP$ replaced by $D^{i+1}$ and only applying Lemma 7.12 in \cite{GT01} in the last step, it is not hard to obtain $\partial^{\alpha} v \in {W}^{k-i,p}_{\locc}(\Omega,U)$ for any $\alpha$ with $\abss{\alpha}=i$, for $i =k-1,\dots,0$.
    \par\noindent \textbf{Part 2.} Full-space elliptic estimate.

    Having established that $v \in W^{k,p}_{\locc}(\Omega,U)$, we are now in the position to show the inequality \eqref{eq:Korn}. 
    Take $\rho$ as above and define again $w=\rho v$. 
   We then write in Fourier space
    \begin{align*}
        \hat w(\xi)=[\proj_{\ker \B(\xi)^\perp}+\proj_{\rmim \C(\xi)}]\hat w(\xi)=
        \B^\dagger(\xi)\B(\xi)\hat w(\xi)+\C^{*\dagger}(\xi)\C^*(\xi)\hat w(\xi),
    \end{align*}
    so that 
    \begin{align*}
        \widehat{D^k w}(\xi)=\B^\dagger\left(\dfrac{\xi}{|\xi|}\right)\widehat{\B w}(\xi)\otimes \left(\dfrac{\xi}{|\xi|}\right)^{\otimes k}+\C^{*\dagger}\left(\dfrac{\xi}{|\xi|}\right)\dfrac{\widehat{\C^* w}(\xi)}{|\xi|^q}\otimes \left(\dfrac{\xi}{|\xi|}\right)^{\otimes k}.
    \end{align*}
    By the H\"ormander--Mikhlin multiplier theorem and boundedness of Calder\'on--Zygmund operators on weighted $L^p$ spaces \cite[Theorem 7.4.6]{Grafakos1}, we have that for any Muckenhoupt weight $\ww \in A_p$
    \begin{align*}
        \|D^k w\|_{L^p(\R^n;\ww)}\leq C(n,p,[\ww]_{A_p})\left(\|\B w\|_{L^p(\R^n;\ww)}+\|\FT^{-1}(|\xi|^{-q}\widehat{\C^*w})\|_{L^p(\R^n;\ww)}\right).
    \end{align*}
    Using Theorem~\ref{thm:HMmultiplier}, we obtain
    \begin{align}\label{eq:HM}
        \normm{D^kw}_{L^p(\R^n;\ww)} \leq C(\normm{\B w}_{L^p(\R^n;\ww)}+\normm{\C^{\ast}w}_{\dot{W}^{-q,p}(\R^n;\ww)})
        \end{align} with $C=C(n,p,q,[\ww]_{A_p})>0$.
    \par \noindent\textbf{Part 3.} Localized weighted estimate and extrapolation. 
    
        In this part, we denote by $\ww$ an arbitrary $A_p$-weight. Since clearly $\normm{\partial^{\alpha}v}_{L^p(\omega;\ww)} \leq \normm{\partial^{\alpha}w}_{L^p(\R^n;\ww)}$, it is sufficient to estimate the right-hand side of the above inequality.  It is easy to see that \[\normm{\B w}_{L^p(\R^n;\ww)} = \normm{\B (\rho v)}_{L^p(\R^n;\ww)} \leq \normm{\B v}_{L^p(\Omega^{\pprime};\ww)}+ C\normm{v}_{W^{k-1,p}(\Omega^{\pprime};\ww)}.\]
        The other term in \eqref{eq:HM} can be controlled as follows: 
        \begin{align}\label{eq:Sobolev_localisation}
        \begin{split} \normm{\C^{\ast}w}_{\dot{W}^{-q,p}(\R^n;\ww)} &\leq C \normm{\C^{\ast}w}_{\dot{W}^{-q,p}(\Omega^{\prime};\ww)}\\  
        &\leq C\normm{\rho \C^{\ast}v}_{\dot{W}^{-q,p}(\Omega^{\prime};\ww)} + C\sum_{i=1}^{l}\normm{d_i[D^i \rho, D^{l-i}v]}_{\dot{W}^{-q,p}(\Omega^{\prime};\ww)}, 
        \end{split}
        \end{align}
        where $\{d_i[\cdot, \cdot]\}_{i=1}^{l}$ are bilinear forms with constant coefficients. 
\par To see the first inequality, take any $\phi \in C_c^{\infty}(\R^n,U)$ and a polynomial $P_{\phi}$ of degree $q-1$ such that $\int_{\Omega^{\prime}}D^i(\phi-P_{\phi})\dif x=0$, $i=0,\dots,q-1$.
Take a cut-off function $\chi\in C_c^{\infty}(\Omega')$ with $\chi \equiv 1$ on $\Omega^{\pprime}$. Notice that $\spt w \subset \Omega''$ and $\deg \C =l \geq q$, then we have 
        \begin{align*}
            \brangle{\C^{\ast}w,\phi} &= (-1)^l \brangle{w,\C \phi}\\
            &= (-1)^l \brangle{w,\C (\chi(\phi-P_{\phi}))} \\
            &= \brangle{\C^{\ast}w,\chi(\phi-P_{\phi})}.
        \end{align*} Set $\ww'=\ww^{-1/(p-1)}$, which lies in $A_{p'}$. Then if $\C^{\ast}w \in \dot{W}^{-q,p}(\Omega',U;\ww)$, by Poincar\'{e}'s inequality on weighted Sobolev spaces \cite[Chapters~1 and~15]{HKM}, there holds
        \begin{align*}
            & \brangle{\C^{\ast}w,\phi} =  \brangle{\C^{\ast}w,\chi(\phi-P_{\phi})} \\
            &\leq \normm{\C^{\ast}w}_{\dot{W}^{-q,p}(\Omega^{\prime};\ww)}\normm{\chi(\phi-P_{\phi})}_{\dot{W}^{q,p'}(\Omega^{\prime};\ww')} \\
            &\leq C\normm{\C^{\ast}w}_{\dot{W}^{-q,p}(\Omega^{\prime};\ww)} \normm{\phi}_{\dot{W}^{q,p'}(\R^n;\ww')},
        \end{align*} where $C=C(n,p,q,[\ww]_{A_p},\Omega',\Omega'')>0$.
        
\par Test the terms on the second line in \eqref{eq:Sobolev_localisation} with any $\varphi \in C_c^{\infty}(\Omega^{\prime},U)$ separately, and similarly we have
        \begin{align*}
            \brangle{\rho \C^{\ast}v, \varphi}& = \brangle{\C^{\ast}v, \rho \varphi} \\
            &\leq \normm{\C^{\ast}v}_{\dot{W}^{-q,p}(\Omega^{\prime};\ww)} \normm{\rho \varphi}_{\dot{W}^{q,p'}(\Omega^{\prime};\ww')} \\
            &\leq  C\normm{\C^{\ast}v}_{\dot{W}^{-q,p}(\Omega^{\prime};\ww)} \normm{\varphi}_{W^{q,p'}(\Omega^{\prime};\ww')} \\
            &\leq C\normm{\C^{\ast}v}_{\dot{W}^{-q,p}(\Omega^{\prime};\ww)} \normm{\varphi}_{\dot{W}^{q,p'}(\Omega^{\prime};\ww')},
        \end{align*} 
        where the last inequality follows from Poincar\'{e}'s inequality and $C=C(n,p,q,[\ww]_{A_p},\omega,\Omega'')>0$;
        \begin{align*}
            \brangle{d_i[D^i\rho,D^{l-i}v],\varphi} &= \brangle{D^{l-i}v, d_i^{\prime}[D^i\rho,\varphi]} \\
            &= (-1)^{l-i-t_i}\brangle{D^{t_i}v, (\di)^{l-i-t_i}d_i^{\prime}[D^i\rho,\varphi]} \\
            &\leq C\normm{D^{t_i}v}_{L^p(\Omega^{\prime};\ww)}\normm{\varphi}_{W^{q,p'}(\Omega^{\prime};\ww')} \\
            &\leq  C\normm{D^{t_i}v}_{L^p(\Omega^{\prime};\ww)}\normm{\varphi}_{\dot{W}^{q,p'}(\Omega^{\prime};\ww')},
        \end{align*} where $\{d_i^{\prime}[\cdot,\cdot]\}_{i=1}^{l}$ are also bilinear forms with constant coefficients, $t_i:=\min\{l-i,k-1\}$, and $C=C(n,p,\C,[\ww]_{A_p},\Omega')>0$. Notice that the terms in the above estimates are allowed to take infinity if the corresponding distributions are not in $\dot{W}^{-q,p}(\Omega',U;\ww)$ or $L^p(\Omega', (U^n)^{t_i};\ww)$, respectively. This proves \eqref{eq:Korn} with $\ww =1$ and also the weighted inequality in \eqref{eq:Korn_ball} by scaling. The constant in the latter only depends on $\ww$ through $[\ww]_{A_p}$, so that we can apply Lemma~\ref{lem:extrapolation} to $(f,g)$ with \[f = \indi_{B_{\tau R}}\abss{D^k v}, \quad g= \indi_{B_R} \brac{\abss{\B v} + \sum_{i=0}^{k-1} \frac{\abss{D^i v}}{R^{k-i}}},\] which gives the modular inequality in \eqref{eq:Korn_ball} with $C_2=(n,p,\B,\C,\tau, \Delta_2(V_p), \Delta_2(V^{\ast}_p))>0$. Since $V_p$ is a fixed Young function related to $p$, we can reduce the dependence on $\Delta_2(V_p)$ and $\Delta_2(V^{\ast}_p)$ to that on $p$ itself.
        \end{proof}

\subsection{Estimates for linear systems}\label{subsec:linearsys}

\par We investigate the regularity of the linear system 
    \[\B^{\ast}(A \B h)=0,\, \C^{\ast}h =0.\] The following estimate is in a similar spirit to those for linear elliptic systems (see, for example, \cite{Gia83} Chapter III) while also requiring the Korn-type inequality \eqref{eq:Korn_ball} since we are dealing with constant rank operators instead of elliptic. We will use this  in Section \ref{sec:partialreg}, where the main partial regularity claim is proved.
\begin{theorem}\label{thm:linear_system}
Suppose that $\B$ and $\C$ are as above, and $\Omega \subset \R^n$ is an open set. If $h \in \mathscr{D}^{\prime}(\Omega,U)$ satisfies the following  
    \begin{equation}\label{eq:linear_system}
        \B^{\ast}(A \B h)=0, \quad \C^{\ast}h =0
    \end{equation} in $\Omega$, where $A\col V \times V \to \R$ is a symmetric linear operator with constant coefficients and satisfies \begin{align*}
    A[z,z] \geq \lambda \abss{z}^2\text{ for }z\in \bigcup_{\xi\in\R^n\setminus\{0\}} \rmim\B(\xi)\quad\text{and}\quad \abss{A[z,w]}\leq \Lambda \abss{z}\abss{w} \mbox{ for }z,w \in V 
    \end{align*}
    with $0<\lambda \leq \Lambda <\infty$. Then we have $h \in C^{\infty}(\Omega,U)$ and the following estimate
    \begin{equation}\label{eq:linear_system_est}
        r^2\sup_{B_r} \abss{D^{k+1}h}+r\sup_{B_r} \abss{D^k h} + \sup_{B_r} \abss{D^{k-1}h} \leq C\fint_{B_{2r}}\abss{D^{k-1} h}\dif x
    \end{equation} for any ball $B_r$ with $B_{2r} \Subset \Omega$ and $0<r<1$, where $C=C(n,\dim V, \frac{\Lambda}{\lambda})>0$.
\end{theorem}

\begin{proof}
Consider the operator $\PP:=\Delta^q\B^*A\B+\C\C^*$, so that $\PP h=0$ (recall that $\B$ and $\C$ have orders $k$ and $l$ respectively and $q=l-k$. We claim that $\PP$ is elliptic, i.e., $\ker \PP(\xi)=\{0\}$ for all $\xi\neq0$.

To see this, let $\xi\neq0$, $u_0\in \ker \PP(\xi)$, and look at
\begin{align*}
    0&=\langle u_0, \PP(\xi)u_0 \rangle=|\xi|^{2q}\langle u_0, \B^*(\xi)A\B(\xi)u_0 \rangle+\langle u_0, \C(\xi)\C^*(\xi)u_0 \rangle\\
    &= |\xi|^{2q} A[\B(\xi)u_0,\B(\xi)u_0]+|\C^*(\xi)u_0|^2\geq  \lambda|\xi|^{2q} |\B(\xi)u_0|^2+|\C^*(\xi)u_0|^2,
\end{align*}
so $u_0\in \ker \B(\xi)\cap \ker \C^*(\xi)=\ker \B(\xi)\cap [\rmim \C(\xi)]^\perp=\ker \B(\xi)\cap [\ker \B(\xi)]^\perp=\{0\}$.

This proves our claim that $\PP$ is elliptic, which implies the smoothness of $h$ in $\Omega$.
    
    \par Now we show the estimate. Take a ball $B_{r}$ such that $B_{2r}\Subset \Omega$, and a $k_0$-homogeneous polynomial $P$ ($k_0:=\max\{0,k-2\}$) such that \[ (D^i P)_{2r} =(D^i h)_{2r}, \ 0\leq i \leq \max\{0,k-2\}.\] The function $h-P(=:\tilde{h})$ satisfies the first system in \eqref{eq:linear_system} as well. Then, for fixed $s,t$ with $r\leq t<s\leq 2r$, we take a cut-off function $\rho \in C_c^{\infty}(B_{2r})$ satisfying 
        \[ 0\leq \rho \leq \chi_{B_s}, \quad \rho \equiv 1 \ \mbox{ in }B_{t}, \quad \mbox{and}\quad \abss{D^i \rho} \leq C(s-t)^{-i}, \ i=1,\dots, k,\] and test the first system in \eqref{eq:linear_system} with $\rho^2 \tilde{h}$:
        \begin{align} \label{eq:lnsystem_test}
        \begin{split}
            0 = \int_{\Omega} A[\B \tilde{h}, \B(\rho^2 \tilde{h})] \dif x= \int_{\Omega} A[\B \tilde{h}, \rho^2 \B \tilde{h} + \sum_{i=1}^k R_i(D^i \rho^2,D^{k-i}\tilde{h})]\dif x\\
            = \int_{\Omega} A[\B (\rho \tilde{h}),\B (\rho \tilde{h})]\dif x + \sum_{i=1}^k \int_{\Omega}(A[R_i(D^i \rho,D^{k-i}\tilde{h}),R_i(D^i \rho,D^{k-i}\tilde{h})] \\
            \hspace{.7cm} -2A[\B (\rho \tilde{h}),R_i(D^i \rho,D^{k-i}\tilde{h})]  +A[\B \tilde{h},R_i(D^i \rho^2,D^{k-i}\tilde{h})]) \dif x,
        \end{split}
        \end{align} where $\{R_i[\cdot, \cdot]\}_{i=1}^k$ are bilinear forms and take values in $V$. By the Plancherel theorem and the ellipticity of $A$ in $\cup_{\xi\in\R^n\setminus\{0\}} \rmim\B(\xi)$, we know that 
        \[ \lambda \int_{B_{s}} \abss{\B(\rho \tilde{h})}^2 \dif x \leq \int_{\Omega} A[\B (\rho \tilde{h}),\B (\rho \tilde{h})]\dif x.  \] Then with \eqref{eq:lnsystem_test} and the boundedness of $A$, we have
       \begin{multline}
            \int_{B_{s}} \abss{\B(\rho \tilde{h})}^2 \dif x \leq \lambda^{-1}\int_{B_s} A[\B (\rho \tilde{h}),\B (\rho \tilde{h})]\dif x\\
            \leq  \frac{\Lambda}{\lambda} \sum_{i=1}^k\int_{B_s} (\abss{R_i(D^i \rho,D^{k-i}\tilde{h})}^2  +2 \abss{\B (\rho \tilde{h})}\abss{R_i(D^i \rho,D^{k-i}\tilde{h})} 
            +\abss{\B \tilde{h}}\abss{R_i(D^i \rho^2,D^{k-i}\tilde{h})})\dif x.
        \end{multline} Young's inequality and the choice of $\rho$ furthermore gives
        \begin{align}
        \begin{split}
            &\ \int_{B_{t}} \abss{\B \tilde{h}}^2 \dif x \leq \int_{B_{s}} \abss{\B(\rho \tilde{h})}^2 \dif x \\
            \leq&\ C(\Lambda/\lambda) \brac{\sum_{i=1}^{k} \int_{B_s}\frac{\abss{D^{k-i}\tilde{h}}^2}{(s-t)^{2i}}\dif x + \int_{B_s\setminus B_t} \abss{\B \tilde{h}}^2 \dif x}.
        \end{split}
        \end{align} Now apply the hole-filling technique, that is, add $C(\Lambda/\lambda)\int_{B_{t}} \abss{\B \tilde{h}}^2 \dif x$ to both sides of the above inequality and divide it by $C(\Lambda/\lambda)+1$, and then we obtain the following inequality by Lemma \ref{lem:iteration} and the choice of $P$:
        \begin{equation}
            \int_{B_{r}} \abss{\B \tilde{h}}^2 \dif x \leq C \sum_{i=0}^{k-1} \int_{B_{2r}} \frac{\abss{D^i \tilde{h}}^2}{r^{2(k-i)}}\dif x \leq C\int_{B_{2r}} \frac{\abss{D^{k-1} \tilde{h}}^2}{{r}^2}\dif x.
        \end{equation} Combining the first inequality in \eqref{eq:Korn_ball}, we can control the $L^2$ integral of $D^k \tilde{h}$ by that of $D^{k-1}\tilde{h}$ as follows
        \begin{align}
        \begin{split}
            &\ \int_{B_{r/2}} \abss{D^k h}^2 \dif x =\int_{B_{r/2}} \abss{D^k \tilde{h}}^2 \dif x \\
            \leq&\ C \brac{\int_{B_r} \abss{\B \tilde{h}}^2 \dif x +\sum_{i=0}^{k-1} \int_{B_{r}} \frac{\abss{D^i \tilde{h}}^2}{r^{2(k-i)}}\dif x} \\
            \leq&\ C\int_{B_{2r}} \frac{\abss{D^{k-1} \tilde{h}}^2}{r^2}\dif x = C\int_{B_{2r}} \frac{\abss{D^{k-1} h}^2}{r^2}\dif x.
        \end{split}
        \end{align}
    \par Notice that any derivative of $h$ of any order also solves the system \ref{eq:linear_system}, and thus we can repeat the above procedure (with a different radius ratio) to $D^i h$, $i = 0,\dots,d-1$ for some integer $d>n/2$, to obtain
        \begin{equation}\label{eq:lnsystem_L2est}
            \int_{B_r} \abss{D^{k+i}h}^2 \dif x \leq \cdots \leq C \int_{B_{2r}} \frac{\abss{D^{k-1}h}^2}{r^{2(i+1)}}\dif x.
        \end{equation}
    Then applying Morrey's inequality to $D^{k-1}\tilde{h}$, we have
        \begin{equation}\label{eq:lnsystem_supest}
            \sup_{B_r} \abss{D^{k-1}h}^2 \leq C\sum_{i=0}^d \fint_{B_r} r^{2i}\abss{D^{k-1+i}h}^2 \dif x \leq C\fint_{B_{2r}}\abss{D^{k-1}h}^2 \dif x.
        \end{equation} An $L^1$-estimate of this type follows from the same argument in Corollary 7.1 and Theorem 7.3 in \cite{Giusti03}:
        \begin{equation}\label{eq:lnsystem_L1est}
             \sup_{B_r} \abss{D^{k-1}h} \leq C\fint_{B_{2r}} \abss{D^{k-1}h} \dif x.
        \end{equation} 
    \par Similar to \eqref{eq:lnsystem_supest}, we can control $D^k h$ by replacing $h$ with $D h$ and also with \eqref{eq:lnsystem_L2est}, \eqref{eq:lnsystem_L1est}:
        \begin{align*}
            & \ r^2\sup_{B_{r}} \abss{D^k h}^2 \leq C \fint_{B_{4r/3}} \abss{D^k h}^2\dif x \\
            \leq &\ C \fint_{B_{5r/3}} \frac{\abss{D^{k-1}h}^2}{r^2}\dif x \leq C \brac{\fint_{B_{2r}} \abss{D^{k-1}h}\dif x}^2.
        \end{align*} The estimate of $D^{k+1}h$ can be done analogously, and the proof is now complete.

\end{proof}

\section{Higher integrability}\label{sec:higher_int}

\par The purpose of this section is to establish higher integrability for generalized minimizers of \eqref{eq:P_Omega_Afree} and \eqref{eq:P_Omega}. We follow the commonly used strategy for proving higher integrability (\cite{Gia83}, Chapter V), which involves first showing a Caccioppoli-type inequality and then the application of Poincar\'{e}'s inequality and Gehring's lemma. However, the Korn-type inequality \eqref{eq:Korn} is required since the constant operator $\B$ is here in the place of the gradient operator.

\subsection{Caccioppoli-type inequalities}\label{subsec:caccioppoli}
We prove a Caccioppoli-type inequality, which applies to non-autonomous integrands and arbitrary constant rank operators.
\begin{proposition}\label{prop:Caccioppoli}
Suppose that the integrand $f$ satisfies assumptions \textnormal{\ref{it:fgrowth}, \ref{it:fcontinuity}} and \textnormal{\ref{it:f_B-qcstrong}$'$}, and that $u \in W^{k-1,p}(\Omega,U)$ such that $\B u\in L^p(\Omega,V)$ is a generalized minimizer of \eqref{eq:P_Omega}, and the map $a\col \Omega \to U$ is a $(k-1)$-polynomial such that $\B a$ is constant with $\abss{\B a}\leq M$ for some $M>0$. Then for any $\tau \in (0,1)$ there exists a constant $C=C(n,\dim V,\B, M,p,L, \ell,\tau)>0$ such that the following estimate holds true for any ball $B(x_0,R)(=:B_R) \Subset \Omega$:
        \begin{multline}\label{eq:BCaccioppoli}
            \int_{B_{\tau R}}\abss{V_p(\B(u-a))}^2 \dif x  \leq C \sum_{i=0}^{k-1} \int_{B_R} \abs{V_p\left(\frac{D^{i}(u-a)}{R^{k-i}}\right)}^2 \dif x\\ 
            +CR\sum_{i=0}^{k-1} \int_{B_R} \brangle{\frac{D^i(u-a)}{R^{k-i}}}^{p-1}\frac{\abss{D^i(u-a)}}{R^{k-i}}\dif x + CR \int_{B_R} \brangle{\B(u-a)}^{p-1}\abss{\B (u-a)} \dif x.
        \end{multline} 
\end{proposition}

\begin{proof}
     Fix a ball $B_R \Subset \Omega$ and $r,s>0$ such that $\frac{R}{2}<r<s<R$. Take a cut off function $\rho \in C_c^{\infty}(B_s)$ with \[ 0\leq \rho \leq 1,\quad \rho \equiv 1 \quad \mbox{in }B_r, \quad\abss{D^i\rho}\leq c_i(s-r)^{-i}, i=1,\dots,k.\] Suppose that the map $a\col \Omega \to V$ is as in the assumption.
    \par Then define $\tilde{f}(z):= f(x_0,\B a +z)-f(x_0,\B a)- f_z(x_0,\B a)\cdot z$, $\tilde{u}:=u-a$, $\varphi:= \rho \tilde{u}$, and $\psi := (1-\rho)\tilde{u} = \tilde{u}-\varphi$. From the strong quasiconvexity of $f$ and the estimate \eqref{eq:pshiftedqc}, we have
    \begin{align*}
        \int_{B_r} \abss{V_p(\B \tilde{u})}^2\dif x &\leq \int_{B_s} \abss{V_p(\B \varphi)}^2\dif x\\
       & \lesssim_{M,p,\ell}\int_{B_s} (f(x_0,\B(\varphi +a)) -f(x_0,\B a)-f_z(x_0,\B a)\cdot \B \varphi) \dif x\\
        &=\int_{B_s} \tilde{f}(\B \varphi)\dif x =\int_{B_s} \tilde{f}(\B \tilde{u}-\B \psi)\dif x\\
        &=\int_{B_s} \tilde{f}(\B \tilde{u})\dif x + \int_{B_s} (\tilde{f}(\B \tilde{u} -\B \psi)-\tilde{f}(\B \tilde{u}))\dif x.
    \end{align*} The definition of $\tilde{f}$ and the minimality of $u$ then imply
    \begin{align*}
        \int_{B_s} \tilde{f}(\B \tilde{u})\dif x &= \int_{B_s}f(x,\B u)\dif x + \int_{B_s} (\tilde{f}(\B \tilde{u}) -f(x,\B u))\dif x \\
        &\leq \int_{B_s}f(x,\B u-\B \varphi)\dif x + \int_{B_s} (\tilde{f}(\B \tilde{u}) -f(x,\B u))\dif x.
    \end{align*} The first term in the second line is further controlled as follows:
    \begin{align*}
       \int_{B_s} f(x,\B u&-\B \varphi)\dif x = \int_{B_s}f(x,\B \psi +\B a) \dif x\\
        &=\int_{B_s} f(x_0,\B \psi + \B a)\dif x + \int_{B_s} (f(x,\B \psi +\B a)- f(x_0,\B \psi + \B a))\dif x\\
        &=\int_{B_s} (\tilde{f}(\B \psi)+f(x_0,\B a) + f_z(x_0,\B a)\cdot \B \psi)\dif x\\ 
        &+ \int_{B_s} (f(x,\B \psi +\B a)-f(x_0,\B \psi +\B a))\dif x.
    \end{align*} Combining the three estimates above, we obtain 
    \begin{align}
        \int_{B_r} &\abss{V_p(\B \tilde{u})}^2\dif x \leq C\int_{B_s}\tilde{f}(\B \varphi)\dif x \notag\\
        &\leq C \left(\int_{B_s} \tilde{f}(\B \psi)\dif x + \int_{B_s}(f(x,\B \psi + \B a) -f(x_0,\B \psi + \B a))\dif x\right. \label{eq:CaccioppoliSupmid1}\\
        &\left.+\int_{B_s}(f(x_0,\B u) -f(x,\B u))\dif x +\int_{B_s} (\tilde{f}(\B \tilde{u}-\B \psi) -\tilde{f}(\B \tilde{u}))\dif x\right) \notag\\
        &=:C(I +II+III+IV).\notag
    \end{align} 
    By \eqref{eq:pshiftedup} we know 
    \begin{align} \label{eq:CaccioppoliSupmid2}
   \begin{split}
        I& \leq C \int_{B_s\setminus B_r}\abss{V_p(\B \psi)}^2\dif x \\
    &\leq C \int_{B_s\setminus B_r}\abss{V_p(\B \tilde{u})}^2 \dif x+C\sum_{i=1}^{k}\int_{B_s}\abs{V_p\brac{\frac{D^{k-i}\tilde{u})}{(s-r)^{i}}}}^2\dif x. 
    \end{split}
    \end{align}
    The terms $II$ and $III$ can be controlled with the continuity of $\partial_z f$ in $x$ as assumed in \ref{it:fcontinuity}: 
    \begin{align}
    &II +III = -\int_{B_s} \int_0^1(\partial_z f(x,\B u-t\B \varphi) -\partial_z f(x_0,\B u-t\B \varphi))\cdot \B \varphi\dif t \notag\\
    &\leq C\int_{B_s}\int_0^1 \abss{x-x_0}(1+\abss{\B u-t\B \varphi}^{2})^{\frac{p-1}{2}}\abss{\B \varphi}\dif x \label{eq:CaccioppoliSupmid3}\\
    &\leq Cs\int_{B_s} (1+\abss{\B \tilde{u}}^2+\abss{\B \varphi}^2)^{\frac{p-1}{2}}\abss{\B \varphi}\dif x \notag\\
    &\leq Cs\int_{B_s} \brangle{\B \tilde{u}}^{p-1}\abss{\B \tilde{u}} \dif x + Cs\sum_{i=1}^{k}\int_{B_s} \brangle{\frac{D^{k-i}\tilde{u}}{(s-r)^{i}}}^{p-1}\frac{\abss{D^{k-i}\tilde{u}}}{(s-r)^{i}}\dif x. \notag
    \end{align} Notice that the integrand in $IV$ vanishes on $B_r$, and by \eqref{eq:pshiftedLip} we have 
    \begin{align}\label{eq:CaccioppoliSupmid4} \begin{split}
    IV &\leq C\int_{B_s\setminus B_r} (1+\abss{\B\psi}^2 + \abss{\B\tilde{u}}^2)^{\frac{p-2}{2}}(\abss{\B\psi} + \abss{\B\tilde{u}})\abss{\B\psi}\dif x\\
    &\leq C \int_{B_s\setminus B_r} \abss{V_p(\B \tilde{u})}^2 \dif x + C\sum_{i=1}^{k}\int_{B_s}\abs{V_p\brac{\frac{D^{k-i}\tilde{u}}{(s-r)^{i}}}}^2\dif x. 
    \end{split}
    \end{align} Then \eqref{eq:CaccioppoliSupmid1}-\eqref{eq:CaccioppoliSupmid4} together imply 
    \begin{multline*}
        \int_{B_r} \abss{V_p(\B \tilde{u})}^2\dif x \leq C\int_{B_s\setminus B_r}\abss{V_p(\B \tilde{u})}^2 \dif x + C\sum_{i=1}^{k}\int_{B_s}\abs{V_p\brac{\frac{D^{k-i}\tilde{u}}{(s-r)^{i}}}}^2\dif x\\
        + Cs\brac{\int_{B_s} \brangle{\B \tilde{u}}^{p-1}\abss{\B \tilde{u}} \dif x +\sum_{i=1}^{k}\int_{B_s} \brangle{\frac{D^{k-i}\tilde{u}}{(s-r)^{i}}}^{p-1}\frac{\abss{D^{k-i}\tilde{u}}}{(s-r)^{i}}\dif x }, 
    \end{multline*} and we add $C\int_{B_r} \abss{V_p(\B \tilde{u})}^2\dif x$ to both sides to get
    \begin{multline*}
        \int_{B_r} \abss{V_p(\B \tilde{u})}^2\dif x \leq \frac{C}{C+1}\int_{B_s}\abss{V_p(\B \tilde{u})}^2 \dif x + \frac{C}{C+1}\sum_{i=1}^{k}\int_{B_s}\abs{V_p\brac{\frac{D^{k-i}\tilde{u}}{(s-r)^{i}}}}^2\dif x \\
        +\frac{C}{C+1}s\brac{\int_{B_s} \brangle{\B \tilde{u}}^{p-1}\abss{\B \tilde{u}} \dif x +\sum_{i=1}^{k}\int_{B_s} \brangle{\frac{D^{k-i}\tilde{u}}{(s-r)^{i}}}^{p-1}\frac{\abss{D^{k-i}\tilde{u}}}{(s-r)^{i}}\dif x } .
    \end{multline*} The desired inequality \eqref{eq:BCaccioppoli} follows from Lemma \ref{lem:iteration} with 
    \begin{align*}
    \Phi(r) &= \int_{B_r}\abss{V_p(\B\tilde{u})}^2\dif x,\\
    \Psi_i(t) &= \int_{B_R}\abs{V_p\brac{\frac{D^{k-i}\tilde{u}}{(s-r)^{i}}}}^2\dif x + R\int_{B_s} \brangle{\frac{D^{k-i}\tilde{u}}{(s-r)^{i}}}^{p-1}\frac{\abss{D^{k-i}\tilde{u}}}{(s-r)^{i}}\dif x,\\
    B &= CR\int_{B_R}\brangle{\B \tilde{u}}^{p-1}\abss{\B \tilde{u}} \dif x, 
    \end{align*}
    which concludes the proof.
\end{proof}
We have the following improved Caccioppoli-type inequality for good minimizers $u$ of \eqref{eq:P_Omega}, which follows from Proposition~\ref{prop:Caccioppoli} and inequality~\eqref{eq:Korn_ball}. 
\begin{proposition}\label{prop:Caccioppoli_Dk}
    Let an integrand $f$ satisfy \textnormal{\ref{it:fgrowth}, \ref{it:fcontinuity}} and \textnormal{\ref{it:fqcstrong}$'$} and $u \in W^{k-1,p}(\Omega,U)$ be such that $\B u\in L^p(\Omega,V)$ is a generalized minimizer of \eqref{eq:P_Omega} satisfying $\C^{\ast}u=0$. Take a map $a\col \Omega \to U$ with $\B a$ being constant and $\abss{\B a}\leq M$ for some $M>0$. Then for any $\tau \in (0,1)$ there exists a constant $c_{M,\tau}=c(n,\dim V,\B, M,p,L,\ell,\tau)>0$ such that the following estimate holds true for any ball $B(x_0,R)(=:B_R) \Subset \Omega$:
        \begin{multline}\label{eq:BCaccioppoli_Dk}
            \int_{B_{\tau R}}\abss{V_p(D^k(u-a))}^2 \dif x  \leq c_{M,\tau}\left( \sum_{i=0}^{k-1} \int_{B_R} \abs{V_p\left(\frac{D^{i}(u-a)}{R^{k-i}}\right)}^2 \dif x \right.\\ 
            \left. +R\sum_{i=0}^{k-1} \int_{B_R} \brangle{\frac{D^i(u-a)}{R^{k-i}}}^{p-1}\frac{\abss{D^i(u-a)}}{R^{k-i}}\dif x + R \int_{B_R} \brangle{\B(u-a)}^{p-1}\abss{\B (u-a)} \dif x \right).
        \end{multline} 
\end{proposition}

\begin{remo}\label{rmk:A_Caccioppoli}
    Notice that Proposition \textnormal{\ref{prop:Caccioppoli}} and \textnormal{\ref{prop:Caccioppoli_Dk}} also hold true for (good) minimizers of \eqref{eq:P_Omega_Afree}, which follows from Remark \textnormal{\ref{rmk:reduction}}. The extra function $\phi$ can be considered as part of the $x$-dependence. 
\end{remo}

\subsection{Higher integrability for minimizers}

\par The Caccioppoli-type inequality \eqref{eq:BCaccioppoli_Dk} implies  higher integrability of minimizers of our variational problems \eqref{eq:P_Omega} and \eqref{eq:P_Omega_Afree}. The proof in the case $1<p<2$ requires the following Poincar\'e inequality with respect to $V$. It will also be combined with the Caccioppoli inequality \eqref{eq:BCaccioppoli_Dk} in the excess decay estimate in Section \ref{sec:partialreg} and can be shown by modifying Theorem 2.4 in \cite{CFM98}.
    \begin{theorem}\label{thm:Vpoincare}
        If $1<p<2$, there exist $\alpha \in (\frac{2}{p}, 2)$ and $\sigma >0$ such that any map $u \in W^{1,p}(B_{3R},U)$ satisfies
        \begin{equation*}
            \sum_{i=0}^{k-1}\brac{\fint_{B_R} \abs{V_p\brac{\frac{D^iu-(D^i u)_R}{R^{k-i}}}}^{2(1+\sigma)}\dif x}^{1/(2(1+\sigma))} \leq C \brac{\fint_{B_{3R}} \abss{V_p (D^k u)}^{\alpha}\dif x}^{1/\alpha}
        \end{equation*} with $C=C(n,p,k,\dim U)>0$.
    \end{theorem}
\par The higher integrability result states as follows:

\begin{proposition}\label{prop:higherint}
    Let an integrand $f$ satisfy \textnormal{\ref{it:fgrowth}, \ref{it:fcontinuity}} and \textnormal{\ref{it:f_B-qcstrong}$'$}, and that $u \in W^{k,p}(\Omega,U)$ be a generalized minimizer of \eqref{eq:P_Omega} satisfying $\C^{\ast}u=0$. Then there exists $\sigma_0 >0$ depending on $n,p$ and $c_{0,1/2}$ in \eqref{eq:BCaccioppoli_Dk} such that $D^k u \in L^{p+\sigma}_{\locc}(\Omega,(U^n)^k)$ for any $\sigma \in (0,\sigma_0)$.
\end{proposition}
\begin{proof}
    \par It is sufficient to derive a weak reverse H\"{o}lder inequality, and then the desired integrability result follows from the generalized Gehring lemma (see, for example, $\S$V.1 in \cite{Gia83}).
    \par In the case where $p\geq 2$, we have \[ \brangle{z}^{p-1}\abss{z} \lesssim \abss{z}+\abss{z}^{p} \lesssim 1+\abss{z}^{p} \] for any vector $z$. In addition, the estimate \ref{it:Vest} in Lemma \ref{lem:Vest} implies that \[ \abss{z}^p \lesssim \abss{V_p(z)}^2 \lesssim 1+\abss{z}^p.\] Fix an arbitrary ball $B_R \Subset \Omega$ with $R\leq 1$ and take a $(k-1)$-polynomial $a$ such that \[(D^i a)_{B_R} =(D^i u)_{B_R}, \quad i=0,\dots, k-1. \] Then from \eqref{eq:BCaccioppoli_Dk} with $\tau = \frac{1}{2}$ we obtain 
    \begin{align*}
        \fint_{B_{R/2}} \abss{D^k u}^p \dif x &\leq  C\brac{\sum_{i=0}^{k-1} \fint_{B_R} \frac{\abss{D^i (u-a)}^p}{R^{p(k-i)}}\dif x  + R\fint_{B_R} \abss{\B (u-a)}^{p}\dif x +1}\\
        &\leq C\brac{ \brac{\fint_{B_R} \abss{D ^ku}^{p_{\ast}}\dif x}^{\frac{p}{p_{\ast}}} + R\fint_{B_R} \abss{\B (u-a)}^{p}\dif x+1},
    \end{align*} where the second line follows from the Sobolev-Poincar\'{e} inequality and $p_{\ast} = \frac{np}{n+p}$. Take $R \ll 1$ such that $CR <1$. Then this weak reverse H\"{o}lder inequality together with the generalized Gehring lemma (\cite{Gia83}, Chapter V, Proposition 1.1) implies the local higher integrability of $D^k u$.
    \par When $1<p<2$, we show a weak reverse H\"{o}lder inequality with respect to $V_p(z)$ instead. Notice that \[\brangle{z}^{p-1}\abss{z} \lesssim \abss{z}+\abss{z}^p \lesssim 1+\abss{V_p(z)}^2. \] Fix an arbitrary ball $B_R$ with $B_{3R} \Subset \Omega$ and $R\leq 1$, take a $(k-1)$-polynomial $a$ as above and let $\tau = \frac{1}{2}$. Then \eqref{eq:BCaccioppoli_Dk} implies
        \begin{align*}
             \fint_{B_{R/2}}\abss{V_p&(D^k u)}^2 \dif x \leq C\left(\sum_{i=0}^{k-1} \fint_{B_R} \abs{V_p\left(\frac{D^{i}(u-a)}{R^{k-i}}\right)}^2 \dif x \right. \\
          &\left.+\sum_{i=0}^{k-1}\fint_{B_R} \brac{1+\abs{V_p\left(\frac{D^i(u-a)}{R^{k-i}} \right)}^2}\dif x +R\fint_{B_R}(1+\abss{V_p(\B (u-a))}^2)\dif x \right)\\
            &\left.\leq C \brac{\brac{\fint_{B_{3R}} \abss{V_p(D^ku)}^{\beta} \dif x}^{\frac{2}{\beta}} +R\fint_{B_R} \abss{V_p(D^k u)}^2\dif x +1},\right.
        \end{align*} where the third line follows from Theorem \ref{thm:Vpoincare} and $\beta \in (\frac{2}{p},2)$. Taking $R\ll 1$ such that $CR<1$ and applying the generalized Gehring lemma again, we obtain the local higher integrability of $V_p(D^k u)$, and thus of $D^k u$.
\end{proof}

\section{Partial regularity}\label{sec:partialreg}
    \par To prove Theorem \ref{thm:main}, it is enough to prove the corresponding excess decay estimate (Proposition \ref{prop:EDE}). These imply the claimed partial H\"{o}lder regularity using a routine argument, which can be found in \cite{Gia83}, Chapter IV. We consider the super-quadratic case $p\geq 2$ and sub-quadratic one $p<2$ separately, and base our proofs on \cite{Evans86} and \cite{CFM98}. 
    \begin{proof}
        [Proof of Theorem~\ref{thm:main}]
        Follows from Propositions \ref{prop:EDE} and a standard argument (see, for instance, Section 9.5 in \cite{Giusti03}), Proposition~\ref{prop:higherint} (for higher integrability), and Proposition~\ref{prop:reduction} (to verify that minimizers can be assumed $\C^*$-free).
    \end{proof}
\begin{proof}
    [Proof of Theorem~\ref{thm:main_A-free} assuming Theorem~\ref{thm:main}]
    Follows from Remarks~\ref{rmk:reduction} and \ref{rmk:A_Caccioppoli}.
\end{proof}
To finish the paper, it remains  to prove Proposition \ref{prop:EDE}.
\par Let $\alpha \in (0,1)$ be a given exponent. We consider the following excess functional:
        \begin{equation}\label{eq:excess}
            \EE(u,x_0,R):= R^{2\alpha}+\fint_{B_R(x_0)}\abss{V_p(D^k u- (D^k u)_{x_0,R})}^2 \dif x.
        \end{equation}
\begin{proposition}\label{prop:EDE}
        Suppose that the integrand $f$ satisfies \textnormal{\ref{it:fgrowth}, \ref{it:fcontinuity}} and \textnormal{\ref{it:f_B-qcstrong}$'$} with $1<p <\infty$, and that $u \in W^{k,p}(\Omega,U)$ is such that $\B u$ a generalized minimizer of \eqref{eq:P_Omega} with $\C^{\ast} u=0$. For any given $M>0$, $\alpha \in (0,1)$ and $\tau \in (0,\frac{1}{4})$, there exists $\varepsilon>0$ such that the following holds true: for any ball $B(\bar{x},R) (=:B_R) \Subset \Omega$,  the conditions 
        \begin{equation*}
            \abss{(D^k u)_{R}}< M, \quad \EE(u,\bar{x},R)<\varepsilon
        \end{equation*} implies
        \begin{equation*}
            \EE(\tau R)\leq C_0 \tau^{2\alpha} \EE(R)
        \end{equation*} for some $C_0=C_0(n,p,\B,\dim V,L, \ell, M)>0$, where $\EE$ is defined as in \eqref{eq:excess}.
    \end{proposition}

    \begin{proof}
        The proposition is proved by contradiction and the proof is divided into four steps.
    ~\paragraph*{\bf Step 1. Contradiction assumption.} Suppose that the proposition is not true for a certain $\tau \in (0,\frac{1}{4})$. Then there exists a sequence of balls $\{B(x_m,r_m)\}_{m\in \N}$ with $B(x_m,r_m) \Subset \Omega$ such that 
    \begin{align}
        &\abss{(D^k u)_{x_m,r_m}}<M, \quad \EE(u,x_m,r_m)=\lambda_m^2 \overset{m \to \infty}{\longrightarrow} 0, \label{eq:HAcontra2}\\
        &\EE(u,x_m,\tau r_m) >C_0 \tau^{2\alpha}\EE(u,x_m, r_m), \label{eq:HAcontra1}
    \end{align} where the constant $C_0$ is to be determined. The assumptions imply $r_m \to 0$. 
    \par Define the following map on $B_1=B(0,1)$ for each $m\in \N$: 
        \begin{equation*}
            v_m(y):= \frac{1}{\lambda_m r^k_m} (u(x_m+r_m y) -r^k_m a_m(y)),
        \end{equation*} where the map $a_m$ on $B_1$ is a $k$-polynomial such that 
        \begin{equation*}
            \int_{B(x_m,r_m)} D^i(u-a_m)\dif x =0, \quad 0\leq i \leq k-1, \quad D^k a_m = (D^k u)_{x_m,r_m}:= A_m.
        \end{equation*} Notice that $\C^{\ast}u =0$ implies $\C^{\ast}a_m=0$. Then we have, by direct calculation and \eqref{eq:HAcontra1},
        \begin{align}
            &D^k v_m = \frac{1}{\lambda_m}(D^k u(x_m+r_my)-A_m), \quad \C^{\ast}v_m =0, \label{eq:HAblowup.03}\\
            &(D^i v_m)_{1}=0, \quad 0\leq i \leq k,\notag\\
            & r_m^{2\alpha}+\fint_{B_1} \abss{V(\lambda_m D^k v_m)}^2 \dif y = \EE(u,x_m,r_m)=\lambda_m^2, \label{eq:HAblowup.01}\\
            & (\tau r_m)^{2\alpha}+\fint_{B_{\tau}} \abss{V(\lambda_m (D^k v_m-(D^k v_m)_{\tau}))}^2 \dif x  >C_0 \lambda_m^2\tau^{2\alpha}. \label{eq:HAblowup.02}
        \end{align}
    \par  In addition, \eqref{eq:HAblowup.01} implies $\fint_{B_1}\abss{D^k v_m}^{p_0} \dif y \leq 1$, where $p_0 := \min\{2,p\}$. Using the choice of $a_m$ and Poincar\'e's inequality, we know \[ \normm{v_m}_{W^{k,p_0}(B_1)} \leq C \normm{D^k v_m}_{L^{p_0}(B_1)} \leq C.\]
    \paragraph*{\bf Step 2. Blow up.} Now we know that $\{v_m\}$ is uniformly bounded in $W^{k,p_0}(B_{1},U)$ with $\C^{\ast} v_m=0$, then up to a subsequence there exists a map $v\in W^{k,p_0}(B_{1})$ such that 
        \begin{equation*}
            \left\{\begin{aligned}
            &v_m \weakto v& &\mbox{in }W^{k,p_0}(B_1),\\
            &\C^{\ast}v =0& &\mbox{in }\mathscr{D}^{\prime}(B_1).
            \end{aligned}\right.
        \end{equation*} In particular, we have the strong convergence $v_m \to v$ in $W^{k-1,p_0}(B_1,U)$. Notice that $\{A_m\}$ is a bounded sequence, so without loss of generality, we can assume that $A_m \to A_0$ for some $A_0 \in \R^{N^k\times n}$ with $\abss{A_0}\leq M$. Similarly, we can assume $x_m \to x_0 \in \Omega$. Then it is easy to see 
        \begin{align*}
            &\abss{(\B v_m)_1} = \abss{(T(D^k v_m))_1} = \abss{T((D^k v_m)_1) }\leq C(\B, M),\\
            &\mbox{and}\quad  \abss{(\B v)_1} = \lim_{m\to \infty}\abss{(\B v_m)_1}  \leq C(\B, M), 
        \end{align*} where $T \col \R^{N^k\times n} \to V$ is a linear operator with $T(D^k u)=\B u$.
    ~\paragraph*{\bf Step 3. The harmonicity of $v$.} This step is to show that $v$ solves the following system on $B_1$ in a weak sense: \[ \B^{\ast}( \partial_z^2 f(x_0,T(A_0))\B v) =0. \] Given any suitable test map $\psi \in C_c^{\infty}(B_1,U)$, we test the Euler-Lagrange system satisfied by $u$ on $B(x_m,r_m)$ with $\psi_m(\cdot):= \psi((\cdot-x_m)/r_m)$:
        \begin{align}
            0 = \fint_{B(x_m,r_m)} \partial_z f(x,\B u)\cdot \B \psi_m \dif x = \fint_{B_1} \partial_z f(x_m+r_my,\lambda_m \B v_m +T(A_m))\cdot \B \psi \dif y \notag\\
            = \fint_{B_1}(\partial_z f(x_m+r_my,\lambda_m \B v_m +T(A_m)) - \partial_z f(x_m+r_m y,T(A_m)))\cdot \B \psi \dif y \hspace{.5cm} \label{eq:SupHAsys1}\\
              +\fint_{B_1}(\partial_z f(x_m+r_my, T(A_m)) - \partial_z f(x_m,T(A_m)))\cdot \B \psi \dif y := \fint_{B_1} I\dif y +\fint_{B_1} II\dif y.\notag
        \end{align} Notice that \[ \fint_{B_1} II\dif y \leq r_m L(1+\abss{T(A_m)}^2)^{\frac{p-1}{2}}\fint_{B_1}\abss{\B \psi}\dif y \leq C(p,L,M)r_m \fint_{B_1}\abss{\B \psi}\dif y \] by the Lipchitz continuity of $\partial_z f(\cdot,z)$ as in \ref{it:fcontinuity}. Since $r_m \leq \lambda_m^{\frac{1}{\alpha}}$, we have 
        \begin{equation}\label{eq:SupHAsys2}
            \lim_{m\to \infty}\frac{1}{\lambda_m} \fint_{B_1} II\dif y =0.
        \end{equation}
    \par To analyze the other term, we split $B_1$ into 
        \begin{align*}
            & E_m^+ := \{y\in B_1 \col \lambda_m \abss{\B v_m} \geq 1\},\\
            & E_m^- := \{y\in B_1 \col \lambda_m \abss{\B v_m} < 1\}.
        \end{align*} The measure of $E_m^+$ is controlled with \eqref{eq:HAblowup.01} and \ref{it:Vest}:
        \[\abss{E_m^+} \leq C\int_{E_m^+} \abss{V(\lambda_m \B v_m)}^2 \dif y \leq C\int_{E_m^+} \abss{V(\lambda_m D^k v_m)}^2 \dif y \leq C\lambda_m^2.\] It also indicates that $\abss{E_m^+} \to 0$ as $m\to \infty$. Thus, with the growth rate of $\partial_z f$ we have
        \begin{align}
            \frac{1}{\lambda_m} \int_{E_m^+}I \dif y &\leq \frac{C}{\lambda_m} \int_{E_m^+} (1+\abss{\lambda_m \B v_m}^{p-1})\dif y \notag \\
            &\leq \frac{C}{\lambda_m}\brac{\abss{E_m^+} +\brac{\int_{E_m^+} \abss{V(\lambda_m D^k v_m)}^2 \dif y}^{\frac{p-1}{p}}\abss{E_m^+}^{\frac{1}{p}}}\label{eq:SupHAsys3}\\
            &\overset{\eqref{eq:HAblowup.01}}{\leq} C \lambda_m \overset{m\to \infty}{\longrightarrow} 0. \notag
        \end{align} Notice that $\lambda_m \B v_m \to 0$ a.e. in $B_1$ up to a subsequence and $\indi_{E_m^-} \to 1$ in $L^2(B_1)$. Then considering the boundedness of the arguments and the regularity of $f$ in \ref{it:fcontinuity}, we use H\"{o}lder's inequality and the dominated convergence theorem to obtain
        \begin{multline}\label{eq:SupHAsys4}
            \frac{1}{\lambda_m}\int_{E_m^-} I \dif y = \int_{E_m^-} \int_0^1 \partial_z^2 f(x_m+r_m y,T(A_m) +t\lambda_m \B v_m)\dif t [\B v_m, \B \psi]\dif y \\
            =\int_{E_m^-} \int_0^1 (\partial_z^2 f(x_m+r_m y,T(A_m) +t\lambda_m \B v_m)-\partial_z^2 f(x_0,T(A_0)))\dif t [\B v_m, \B \psi]\dif y  \\
            + \int_{E_m^-}\partial_z^2 f(x_0,T(A_0))[\B v_m, \B \psi]\dif y \overset{m\to \infty}{\longrightarrow} \int_{B_1} \partial_z^2 f(x_0,T(A_0))[\B v, \B \psi]\dif y 
        \end{multline} With \eqref{eq:SupHAsys1}-\eqref{eq:SupHAsys4} we can conclude the following equality \[ \fint_{B_1} \partial_z^2f(x_0,T(A_0))[\B v,\B \psi] \dif x=0,\] which indicates that $v$ solves the following system
        \begin{equation}\label{eq:HAsystem}
            \left\{ \begin{aligned}
            &\B^{\ast}(\partial_z^2 f(x_0,T(A_0))\B v) =0\\
            &\C^{\ast} v=0
            \end{aligned}\right. 
        \end{equation} on the ball $B_{1}$. 
      
    ~\paragraph*{\bf Step 4. Excess decay estimate.} From Theorem \ref{thm:linear_system}, we know that $v \in C^{\infty}(B_1,U)$ and satisfies the following inequality
        \begin{equation}\label{eq:EDharmonicest}
            \sup_{B_{\frac{1}{2}}} \abss{D^{k+1} v}\leq C\fint_{B_{1}} \abss{D^k v - b}\dif x \leq C \brac{\fint_{B_{1}} \abss{D^k v - b}^{p_0}\dif x}^{\frac{1}{p_0}}
        \end{equation} with some constant $C>0$ and any constant vector $b$. For the given $\tau \in (0,\frac{1}{4})$, take $k$-polynomials $\{b_m\}$ such that 
        \begin{align}
        &\int_{B_{2\tau}} D^{i}(v_m-b_m)\dif x =0, \quad 0\leq i \leq k-1, \label{eq:k-affine-ave}\\
        &\quad \mbox{and} \quad D^k b_m = (D^k v_m)_{2{\tau}}. \label{eq:k-affine}
        \end{align} By the convergence $v_m \weakto v$ in $W^{k,2}(B_{1},U)$, we have $b_m$ converges to some $b$ pointwise with 
        \begin{align}
        &\int_{B_{2\tau}} D^{i}(v-b) \dif x =0, \quad 0\leq i \leq k-1, \label{eq:Supk-lim-ave}\\
        &\quad \mbox{and} \quad D^k b = (D^k v)_{2{\tau}}.\label{eq:Supk-lim}
        \end{align}
    \par With \eqref{eq:EDharmonicest}, we can control the mean oscillation of $D^k v$ on $B_{2\tau}$ as follows:
        \begin{align}
             &\fint_{B_{2\tau}} \abss{D^kv-(D^k v)_{2\tau}}^2 \dif x \leq 16\tau^2 \sup_{B_{2\tau}}\abss{D^{k+1} v}^2 \notag\\
            &\leq \ C\tau^2 \brac{\fint_{B_1}\abss{D^k v -(D^k v)_{1}}^{p_0} \dif x}^{2/p_0} \label{eq:HAharexcess_deg1L2}\\
            &\leq\ C\tau^2 \liminf_{m\to \infty} \brac{\fint_{B_1}\abss{D^k v_m -(D^k v_m)_{1}}^{p_0} \dif x}^{2/p_0} \leq C\tau ^2,\notag
        \end{align} where the third line is by the lower semicontinuity of the integral $\int \abss{\cdot}^{p_0}$, the weak convergence $D^{k}v_m \weakto D^{k}v$ in $L^{p_0}(B_1,(U^n)^k)$ and the boundedness of $\{D^k v_m\}$ in $L^{p_0}(B_1,(U^n)^k)$.
   \par The choice of $b_m$ and the fact $(\B v_m)_1 =0$ imply 
        \begin{align*} 
        \lambda_m\abss{D^k b_m}  &=\lambda_m\abss{(D^k v_m)_{2\tau}} =\lambda_m \abss{(D^k v_m)_{2\tau} -(D^k v_m)_1}  \\
        &\leq \frac{1}{(2\tau)^n} \brac{\fint_{B_1}\abss{\lambda_m D^k v_m -(D^k v_m)_1}^{p_0} \dif x}^{\frac{1}{p_0}} \\
        &\leq \frac{C}{(2\tau)^n} \brac{\fint_{B_1}\abss{V(\lambda_m D^k v_m -(D^k v_m)_1)}^{2} \dif x}^{\frac{1}{p_0}}
        \leq \frac{\varepsilon^{1/p_0}}{(2\tau)^{n}} <1 
        \end{align*} when $\varepsilon <(2\tau)^{p_0 n}$. It is then easy to see that \[ \lambda_m (\B v_m(y) -\B b_m(y)) = \B u(x_m+r_m y)-T(A_m) -\lambda_m (\B v_m)_{2\tau}\] with $\abss{T(A_m) +\lambda_m (\B v_m)_{2\tau}}\leq C(\B, M)$. Then \ref{it:Vsum} and \ref{it:Vcvx} in Lemma \ref{lem:Vest}, and Proposition \ref{prop:Caccioppoli_Dk} imply
        \begin{align*}
           &\fint_{B_{\tau}} \abss{V(\lambda_m (D^kv_m -(D^k v_m)_{\tau}))}^2 \dif x \leq C\fint_{B_{\tau}} \abss{V(\lambda_m (D^k(v_m -b_m)))}^2 \dif x \\
           & \hspace{2.8cm}\leq C \sum_{i=0}^{k-1}\fint_{B_{2\tau}} \abs{V\brac{\frac{\lambda_mD^i(v_m-b_m)}{(2\tau)^{k-i}}}}^2 \dif x \\
           & \hspace{2.8cm}+ C\tau r_m \sum_{i=0}^{k-1}\fint_{B_{2\tau}}\brangle{\frac{\lambda_m D^i(v_m-b_m)}{(2\tau)^{k-i}}}^{p-1}\abs{\frac{\lambda_m D^i(v_m-b_m)}{(2\tau)^{k-i}}}\dif x \\
           & \hspace{2.8cm}+C\tau r_m \fint_{2\tau} \brangle{\lambda_m \B(v_m-b_m)}^{p-1}\abss{\lambda_m \B(v_m-b_m)} \dif x =:I+II+III.
        \end{align*} The term $II$ can be estimated as follows
        \begin{align*}
            II &\leq C\tau r_m  \sum_{i=0}^{k-1} \fint_{B_{2\tau}} \brac{\frac{\lambda_m\abss{D^i(v_m-b_m)}}{(2\tau)^{k-i}}+\frac{\lambda_m^p\abss{D^i(v_m-b_m)}^p}{(2\tau)^{p(k-i)}}} \dif x  \\
            &= C\tau r_m  \fint_{B_{2\tau}} (\lambda_m\abss{D^k v_m -(D^k v_m)_{2\tau}}+\lambda_m^p\abss{D^k v_m -(D^k v_m)_{2\tau}}^p)\dif x  \\
            &\leq C\tau^{1-2n}  r_m  \brac{\brac{\fint_{B_{1}} \lambda_m^{p_0}\abss{D^k v_m}^{p_0}\dif x}^{\frac{1}{p_0}} +\fint_{B_{1}} \lambda_m^p\abss{D^k v_m}^p\dif x},
        \end{align*} where the second line follows from Poincar\'{e}'s inequality and the choice of $\{b_m\}$. When $p\geq 2$, by \eqref{eq:HAblowup.01} we know 
        \begin{equation}\label{eq:EDEfinal1}
        II \leq C \tau^{1-2n}r_m (\lambda_m + \lambda_m^2) \leq C \tau^{1-2n}r_m \lambda_m; 
        \end{equation} and if $1<p<2$ it follows from the boundedness of $\{v_m\}$ in $W^{k,p}(B_1,U)$ that
        \begin{equation} \label{eq:EDEfinal2}
        II \leq C \tau^{1-2n}r_m (\lambda_m + \lambda_m^p) \leq C \tau^{1-2n}r_m \lambda_m. 
        \end{equation} The term $III$ can be estimated similarly:
        \begin{equation}\label{eq:EDEfinal3}
        III \leq C \tau^{1-2n}r_m \lambda_m. 
        \end{equation}
    \par To estimate $I$, we control each term ($i=0,\dots,k-1$) with \ref{it:Vsum} in Lemma \ref{lem:Vest}  
        \begin{align*}
            \fint_{B_{2\tau}} \abs{V\brac{\frac{\lambda_m D^i \tilde{v}_m}{(2\tau)^{k-i}}}}^2 \dif x 
            &\leq C \left(\fint_{B_{2\tau}} \abs{V\brac{\frac{\lambda_m D^i (\tilde{v}_m-\tilde{v})}{(2\tau)^{k-i}}}}^2  \dif x \right. \\
            &+\left.\fint_{B_{2\tau}} \abs{V\brac{\frac{\lambda_m D^i \tilde{v}}{(2\tau)^{k-i}}}}^2 \dif x\right):=I_1 +I_2 ,
        \end{align*} where $\tilde{v}_m = v_m-b_m$ and $\tilde{v}=v-b$. It is easy to see $\tilde{v}_m \weakto \tilde{v}$ in $W^{k,p_0}(B_1,U)$. The estimate \ref{it:Vest} in Lemma \ref{lem:Vest}, Poincar\'{e}'s inequality and \eqref{eq:HAharexcess_deg1L2} imply
        \begin{equation}\label{eq:EDEfinal4}
            I_2 \leq C\lambda_m^2 \fint_{B_{2\tau}} \frac{\abss{D^i \tilde{v}}^2}{(2\tau)^{2(k-i)}} \dif x \leq C\lambda_m^2 \fint_{B_{2\tau}} \abss{D^k v-(D^k v)_{2\tau}}^2 \dif x \leq C \lambda_m^2 \tau^2.
        \end{equation}
        Take $\theta \in (0,1)$ such that \[ \frac{1}{2} = \theta + \frac{1-\theta}{2(1+\sigma)},\] where $\sigma$ is as in Theorem \ref{thm:Vpoincare}. Then with H\"{o}lder's inequality we have
        \begin{align*}
            I_1 &\leq C\brac{\fint_{B_{2\tau}} \abs{V\brac{\frac{\lambda_m D^i (\tilde{v}_m-\tilde{v})}{(2\tau)^{k-i}}}}  \dif x}^{2\theta} \brac{\fint_{B_{2\tau}} \abs{V\brac{\frac{\lambda_m D^i (\tilde{v}_m-\tilde{v})}{(2\tau)^{k-i}}}}^{2(1+\sigma)}  \dif x }^{\frac{1-\theta}{1+\sigma}} \\
            &\eqqcolon E_1^{2\theta}\cdot E_2^{1-\theta}.
        \end{align*} Apply \ref{it:Vest} in Lemma \ref{lem:Vest} again to obtain
        \[ E_1 \leq C \lambda_m \fint_{B_{2\tau}} \frac{\abss{D^i(\tilde{v}_m-\tilde{v})}}{(2\tau)^{k-i}} \dif x = \lambda_m \tau^{-k}\cdot o(1) \] as $m \to \infty$. The term $E_2$ can be controlled with Theorem \ref{thm:Vpoincare}:
        \begin{align*}
            E_2 &\leq C \fint_{B_{6\tau}} \abss{V(\lambda_m D^k(\tilde{v}_m-\tilde{v})}^2 \dif x\\
            &\leq C\fint_{B_{6\tau}} (\abss{V(\lambda_m D^k v_m)}^2 + \abss{V(\lambda_m D^k (b_m-b))}^2 +\abss{V(\lambda_m D^k v)}^2) \dif x \\
            &\leq C\brac{\tau^{-n} \fint_{B_1} \abss{V(\lambda_m D^k v_m)}^2 \dif x + \lambda_m^2 \fint_{B_{6\tau}} \abss{D^k (b_m-b)}^2 \dif x + \lambda_m^2\fint_{B_{6\tau}} \abss{D^k v}^2 \dif x}\\
            &\leq C\lambda_m^2 (\tau^{-n} +o(1)) + C\lambda_m^2 \fint_{B_{6\tau}} \abss{D^k v}^2 \dif x.
        \end{align*} With \eqref{eq:EDharmonicest} it is easy to obtain
        \begin{align*}
        \fint_{B_{6\tau}} \abss{D^k v}^2 \dif x &\leq 2\fint_{B_{6\tau}} \abss{D^k v-(D^k v)_{\frac{1}{2}}}^2 \dif x + 2\abss{(D^k v)_{\frac{1}{2}}}^2 \\
        &\leq C \brac{\fint_{B_1} \abss{D^k v}^{p_0} \dif x}^{\frac{2}{p_0}} \leq C,
        \end{align*} which helps us conclude 
        \begin{equation}\label{eq:EDEfinal5}
        I_1 \leq \lambda_m^2 \tau^{-2\theta k}(\tau^{-n}+1)^{1-\theta}\cdot o(1) 
        \end{equation} as $m\to \infty$. Combining \eqref{eq:EDEfinal1}-\eqref{eq:EDEfinal5} and $r_m<\lambda_m^{\frac{1}{\alpha}}$, we have 
        \begin{multline}
            \fint_{B_{\tau}} \abss{V(\lambda_m D^k v_m-(D^k v_m)_{\tau})}^2 \dif x \\
            \leq C\lambda_m^2(\tau^{1-2n}\lambda_m^{\frac{1}{\alpha}-1} +\tau^{-2\theta k}(\tau^{-n}+1)^{\theta}\cdot o(1)) \leq C_1 \lambda_m^2 \tau^{2\alpha}
        \end{multline} when $\varepsilon (\geq \lambda_m^2)$ is taken to be small enough and $m$ is large enough. This contradicts \eqref{eq:HAblowup.02} if we take $C_0>C_1+1$. 
     The proof of the proposition is now complete.
    \end{proof}


\printbibliography
\end{document}